\documentclass[10pt]{amsart}

\usepackage{verbatim}
\usepackage{graphicx}
\usepackage{amssymb}
\usepackage{hyperref}
\usepackage{enumerate}
\usepackage{enumitem}
\usepackage{multicol}

\usepackage{color}

\usepackage{setspace}
\usepackage{mathrsfs}
\newcommand{\gNorm}[1]{{\left\|\kern-0.24ex\left|{ #1 } \right|\kern-0.24ex\right\|}}

\renewcommand{\bar}{\overline}

\renewcommand{\epsilon}{\varepsilon}

\newcommand{\B}{\mathcal{H}}

\newcommand{\D}{d}

\newcommand{\phg}{\textup{\normalfont phg}}

\renewcommand{\Re}{\mathrm{Re}}
\renewcommand{\setminus}{\!\smallsetminus\!}

\newcommand{\Defn}[1]{{\boldmath\it\bfseries #1}}

\DeclareMathOperator{\R}{R}
\DeclareMathOperator{\grad}{grad}
\DeclareMathOperator{\Hess}{Hess}
\DeclareMathOperator{\Div}{div}
\DeclareMathOperator{\tr}{tr}
\DeclareMathOperator{\supp}{supp}

\theoremstyle{plain}
\newtheorem{theorem}{Theorem}

\newtheorem{lemma}[theorem]{Lemma}
\newtheorem{proposition}[theorem]{Proposition}

\newtheorem{remark}[theorem]{Remark}

\newtheorem{assumption}{Assumption}

\numberwithin{equation}{section}

\title[Density of shear-free data]{Smoothly compactifiable shear-free hyperboloidal data is dense in the physical topology}
\author[Allen \& Stavrov Allen]{Paul T.~Allen and Iva Stavrov Allen}

\email{ptallen@lclark.edu, istavrov@lclark.edu}
\address{Department of Mathematical Sciences, Lewis \& Clark College}
\date{\today}

\begin{document}

\maketitle

\begin{abstract}
We show that any polyhomogeneous asymptotically hyperbolic constant-mean-curvature solution to the vacuum Einstein constraint equations can be approximated,  arbitrarily closely in H\"older norms determined by the physical metric, by shear-free smoothly conformally compact vacuum initial data.
\end{abstract}

%%%%%%%%%%%

%------------
\section*{Introduction}
In the study of asymptotically flat (or asymptotically simple) spacetimes,  
initial data corresponding to spacelike slices extending towards null infinity has asymptotically hyperbolic geometry.
Lars Andersson and Piotr Chru\'sciel, building on their work with Helmut Friedrich \cite{AnderssonChruscielFriedrich}, construct in  \cite{AnderssonChrusciel-Dissertationes}  a large number of constant-mean-curvature (CMC)  vacuum initial data sets  with asymptotically hyperbolic geometry using the conformal method of Yvonne Choquet-Bruhat, Andr\'e Lichnerowicz, and James York.
In the work \cite{AnderssonChrusciel-Dissertationes}, particular attention is paid to the regularity of solutions at the conformal boundary.
Data constructed in \cite{AnderssonChrusciel-Dissertationes} typically admits a $C^2$, but not $C^3$ conformal compactification.
In particular, they showed that data which is smooth in the interior ``physical'' manifold is typically polyhomogeneous, rather than smooth, at the conformal boundary.

In their detailed analysis \cite{AnderssonChrusciel-Obstructions}, Andersson and Chru\'sciel show  that initial data must satisfy the shear-free condition along the conformal boundary (see \S\ref{SF:Define}) in order for any resulting spacetime geometry to admit a $C^2$ conformal compactification.
This suggests that one might require the shear-free condition hold in order for a solution to the Einstein constraint equations to be ``admissible'' in the asymptotically hyperbolic setting.
Thus we refer to initial data satisfying the shear-free condition as \Defn{hyperboloidal}, distinguished among those solutions to the constraint equations having asymptotically hyperbolic geometry.
Our recent work  \cite{AHEM-Preliminary}, joint with  James Isenberg and John M.~Lee, contains a systematic study of CMC hyperboloidal initial data, including a parametrization of all such data in the ``weakly asymptotically hyperbolic'' setting (see also \cite{WAH-Preliminary}).

This is not, however, the end of the story.
Even if one restricts attention to shear-free data, the initial data constructed in \cite{AnderssonChrusciel-Dissertationes} and \cite{AHEM-Preliminary} may not be sufficiently regular at the conformal boundary to obtain a spacetime development admitting conformal compactification.
For example, the  existing evolution theorems of Helmut Friedrich \cite{Friedrich-ConformalFieldEquations}, \cite{Friedrich-StaticRadiative}, etc., all require more regularity of the conformal compactification.
(The regularity issue is not unrelated to the shear-free condition: Andersson, Chru\'sciel, and Friedrich show in \cite{AnderssonChruscielFriedrich} that initial data, constructed from smooth ``free data'' using the conformal method (see \S\ref{ConformalMethod}), with pure-trace extrinsic curvature is shear-free if and only if it is smoothly conformally compact.)

In addition to issues of regularity, one may be concerned about whether the collection of hyperboloidal data is sufficiently general for modeling a wide variety of physical situations.

Here we address these issues by showing that any polyhomogeneous asymptotically hyperbolic CMC solution to vacuum constraint equations can be approximated,  arbitrarily closely in H\"older norms determined by the physical (non-compactified) spatial metric, by hyperboloidal (i.e.~shear-free) vacuum initial data that is smoothly conformally compact.
In the case that the conformal boundary is a $2$-sphere, the work of \cite{Friedrich-ConformalFieldEquations} implies that the approximating data has a spacetime development admitting a smooth conformal infinity.

There are a number of ways in which one might interpret our result.
From the perspective of modeling isolated gravitational systems, it is an indication that some version of Bondi-Sachs-Penrose approach to using conformal compactness for studying asymptotically flat spacetimes is feasible for studying a large class of physical systems.
However, it is also an indication that the H\"older topology determined by the physical metric is insufficiently strong for studying the conformal boundary of asymptotically hyperbolic initial data sets.
(For example, it is observed in \cite{AnderssonChrusciel-Obstructions} that among the initial data constructed  in \cite{AnderssonChrusciel-Dissertationes} from smooth ``free data''  by means of the conformal method, the shear-free condition does not hold generically with respect to the $C^\infty$ topology determined by the conformally compactified metric.)
Indeed, it was the approximation result here that motivated several of the results in \cite{AHEM-Preliminary}, where continuity of the conformal method for construction of solutions to the constraint equations is established in a topology strong enough to detect the shear-free condition.

%---------------
\section{Discussion of main result}

Here we present a discussion of the details needed in order to make precise our approximation result.
As we make use of several results from \cite{Lee-FredholmOperators}, \cite{WAH-Preliminary}, and \cite{AHEM-Preliminary}, we maintain conventions similar to the conventions in those works.

\subsection{Asymptotically hyperbolic initial data}
Let $M$ be the interior of a smooth three-dimensional compact manifold $\bar M$ having boundary $\partial M$.
We say that a smooth function $\rho\colon \bar M \to [0,\infty)$ is a  \Defn{defining function} if $\rho^{-1}(0) = \partial M$ and if $\D\rho\neq 0$ on $\partial M$.
A metric $g$ on $M$ is said to be \Defn{$C^k$ conformally compact} if $\bar g:= \rho^2 g$ extends to a metric of class $C^k$ on $\bar M$ for one, and hence all, smooth defining functions $\rho$.
A $C^2$ conformally compact metric $g$ is \Defn{asymptotically hyperbolic} if $|\D\rho|_{\bar g} =1$ along $\partial M$ for one, and hence all, smooth defining functions $\rho$.
The sectional curvatures of such metrics approach $-1$ as $\rho \to 0$; see \cite{WAH-Preliminary} for generalizations of this definition.

A vacuum initial data set $(g,K)$ consists of a Riemannian metric $g$ and symmetric covariant $2$-tensor $K$, both defined on $M$ and satisfying  the vacuum Einstein constraint equations
\begin{subequations}
\label{Constraints}

\begin{equation}
\label{HamiltonianConstraint}
R[g] -|K|^2_g + (\tr_g K)^2 =0,
\end{equation}

\begin{equation}
\label{MomentumConstraint}
\Div_g K-\D(\tr_g K) =0.
\end{equation}
\end{subequations}
It is convenient to introduce the notation $\tau = \tr_gK$ for the trace of $K$ and $\Sigma = K - \frac13\tau g$ for the traceless part of $K$.
We say that \Defn{$(g,K)$ is an asymptotically hyperbolic initial data set} if $g$ is asymptotically hyperbolic and if the tensor $\bar\Sigma = \rho\Sigma$ extends to a $C^1$ tensor field on $\bar M$.
Such a data set is said to be \Defn{smoothly conformally compact} if for any defining function $\rho$ the tensor fields $\bar g=\rho^2 g$ and $\bar\Sigma=\rho \Sigma$ extend smoothly to $\bar M$.
We note that there exist ``weakly asymptotically hyperbolic'' solutions to \eqref{Constraints}, satisfying less stringent regularity conditions; see \cite{AHEM-Preliminary}.

Asymptotically hyperbolic data sets may be viewed as intersecting future null infinity in the asymptotically flat spacetime containing a future development of the data set; we refer the reader to \cite{AHEM-Preliminary}, and the references therein, for a more detailed discussion of  asymptotically hyperbolic initial data sets and asymptotic flatness.

The formula for the change of scalar curvature under conformal deformation, together with \eqref{HamiltonianConstraint}, implies
\begin{equation}
\label{LichRho}
4\rho \Delta_{\bar g}\rho + (R[\bar g] - |\bar\Sigma|^2_{\bar g}) \rho^2 + 6\left( \frac{\tau^2}{9} - |\D\rho|^2_{\bar g}\right) =0,
\end{equation}
where our sign convention on the scalar Laplace operator is $\Delta_{\bar g} = \tr_{\bar g}\Hess_{\bar g}$.
Evaluating \eqref{LichRho} at $\rho =0$ we find that $\tau^2 = 9$ along $\partial M$.
Thus in the constant-mean-curvature (CMC) setting we have $\tau = \pm 3$, with the sign indicating whether the initial data intersects future or past null infinity (relative to the notion of ``future'' determined by $K$).
Henceforth we restrict attention to the CMC case and set $\tau = -3$, which (due to our sign convention for $K$) corresponds to future null infinity; see the discussion in \cite{AHEM-Preliminary}.
Note that when $\tau =-3$ the constraint equations \eqref{Constraints} reduce to
\begin{equation}
\label{CMC-constraints}
R[g]-|\Sigma|^2_g + 6 =0
\quad\text{ and }\quad
\Div_g\Sigma =0.
\end{equation}

%-------
\subsection{The shear-free condition}\label{SF:Define}
While any sufficiently regular solution to the Einstein constraint equations \eqref{Constraints} gives rise to some spacetime development thereof (see \cite{ChoquetBruhat-GRBook} and the references therein), it was shown in \cite{AnderssonChrusciel-Obstructions} that the development of an asymptotically hyperbolic initial data set admits a conformal compactification along future null infinity only if  the \Defn{shear-free condition}
\begin{equation}
\label{FirstShearFree}
\left[\Hess_{\bar g}\rho -\frac13(\Delta_{\bar g}\rho)\bar g -\bar\Sigma\right]_{\partial M} =0
\end{equation}
holds.
We say that an asymptotically hyperbolic initial data set is a \Defn{hyperboloidal initial data set} if \eqref{FirstShearFree} holds.

%------
\subsection{The conformal method}
\label{ConformalMethod}
The existence of asymptotically hyperbolic initial data sets is addressed in \cite{AnderssonChruscielFriedrich} and \cite{AnderssonChrusciel-Dissertationes}.
The existence of hyperboloidal data is discussed in \cite{AHEM-Preliminary}.
All these works make use of the conformal method, which we now describe.

We first introduce the \Defn{conformal Killing operator} $\mathcal D_g$, which maps vector fields to trace-free symmetric covariant $2$-tensors by 
\begin{equation}
\label{DefineCK}
\mathcal D_g W = \frac12 \mathcal L_W g - \frac13( \Div_g W) g.
\end{equation}
The formal $L^2$ adjoint $\mathcal D_g^*$ is given by $\mathcal D_g^* T = -(\Div_g T)^\sharp$, and can be used to construct the self-adoint, elliptic operator $L_g := \mathcal D_g^* \mathcal D_g$, which is called the \Defn{vector Laplacian}.

In the CMC setting, with $\tau = -3$, the conformal method seeks a solution $(g,K)$ to \eqref{Constraints} of the form
\begin{subequations}
\label{Conformal}
\begin{align}
\label{Conformal-g}
g &= \phi^4 \lambda
\\
\label{Conformal-K}
K&= \phi^{-2}\left( \mu + \mathcal D_\lambda W\right) - \phi^4 \lambda,
\end{align}
\end{subequations}
for some Riemannian metric $\lambda$, symmetric covariant $2$-tensor field $\mu$, vector field $W$, and positive function $\phi$.
Replacing $g$ and $K$ in \eqref{Constraints} by the expressions in \eqref{Conformal}, we find that the constraints \eqref{Constraints} are satisfied if $W$ and $\phi$ satisfy the elliptic system
\begin{subequations}
\label{ConformalConstraints}
\begin{align}
\label{ConformalMomentumConstraint}
L_\lambda W &= -\Div_\lambda \mu
\\
\label{Lich}
\Delta_\lambda\phi &= \frac18 \R[\lambda] \phi - \frac18 |\mu + \mathcal D_\lambda W|^2_\lambda \phi^{-7} + \frac34 \phi^5.
\end{align}
\end{subequations}
Thus if $\lambda$ and $\mu$ are specified, it remains only to solve \eqref{ConformalConstraints} in order to obtain a solution to \eqref{Constraints}.
We make use of the nomenclature of \cite{AHEM-Preliminary} and refer to $(\lambda,\mu)$ as a \Defn{free data set}.

If $\lambda$ is an asymptotically hyperbolic metric on $M$, then $g = \phi^4 \lambda$ is an asymptotically hyperbolic metric provided $\phi \in C^2(\bar M)$ and $\phi =1$ along $\partial M$.
If $\phi$ satisfies these conditions, then metric $g$ and tensor $K$ given by \eqref{Conformal} form an asymptotically hyperbolic data set provided $\rho\mu$ and $\rho \mathcal D_\lambda W$ extend to $C^1$ tensor fields on $\bar M$.

We remark that the conformal method, as described above, does not necessarily yield hyperboidal data (i.e.,~data satisfying the shear-free condition).
However, with appropriately constructed free data, one can ensure that the resulting initial data does in fact satisfy the shear-free condition; see \cite{AHEM-Preliminary}.

%-----------
\subsection{Polyhomogeneous data}\label{polyhomog:sec}
The two works \cite{AnderssonChruscielFriedrich} and \cite{AnderssonChrusciel-Dissertationes}, where large classes of asymptotically hyperbolic initial data are constructed, contain detailed analyses of the regularity of solutions at $\partial M$ and show the following: Even if the free data $\lambda$ and $\mu$ are smoothly conformally compact, the solutions $W$ and $\phi$ to \eqref{ConformalConstraints} need not give rise to smoothly conformally compactifiable fields $g$ and $K$.
Rather, the resulting metric $g$ and tensor field $K$ admit formal expansions at $\partial M$ given, in terms of an arbitrary smooth defining function $\rho$, by
\begin{subequations}
\label{PhysicalAsymptotics}
\begin{align}
\label{PhysicalAymptotics-g}
\bar g &\sim \bar g_0 + \sum_{i=0}^{\infty} \sum_{p=0}^{p_i} \rho^{s_i} (\log\rho)^{p} \bar g_{ip},
\\
\label{PhysicalAymptotics-Sigma}
\bar\Sigma &\sim \bar\Sigma_0 + \sum_{i=0}^{\infty}\sum_{q=0}^{q_i} \rho^{t_i} (\log\rho)^{q} \bar\Sigma_{iq},
\end{align}
\end{subequations}
where the barred terms are smooth tensor fields.
Tensor fields which admit such expansions are called ``polyhomogeneous.''
We remark that a number of closely-related definitions of polyhomogeneous tensor fields exist in the literature; see \S\ref{section:boundary} below for a precise definition of the notion of polyhomogenity used here.

The asymptotic expansions of the polyhomogeneous data constructed in \cite{AnderssonChrusciel-Dissertationes} take the form \eqref{PhysicalAsymptotics} with $\Re{(s_0)} >2$ and $\Re{(q_0)} >1$.
Thus, letting $C^k_\phg(\bar M)$ denote the collection of polyhomogeneous tensor fields on $M$ which extend to fields of class $C^k$ on $\bar M$, we have $\bar g\in C^2_\phg(\bar M)$ and $\bar\Sigma\in C^1_\phg(\bar M)$.
The polyhomogeneous hyperboloidal data sets constructed in \cite{AHEM-Preliminary} also have this regularity.

%---------------
\subsection{The approximation theorem}

We now give a careful statement of our main result.

\begin{theorem}
\label{Density}
Suppose that $(g,K)$ is a polyhomogeneous asymptotically hyperbolic vacuum initial data set on $3$-manifold $M$.
Then there exists a family $(g_\epsilon, K_\epsilon)$ of solutions to the vacuum Einstein constraint equations \eqref{Constraints}, defined for sufficiently small $\epsilon >0$,  such that 
\begin{enumerate}
\item each initial data set is hyperboloidal, meaning that each $(M,g_\epsilon)$ is asymptotically hyperbolic, that $(g_\epsilon, K_\epsilon)$ each satisfy the constraint equations \eqref{Constraints}, and that  $(g_\epsilon, K_\epsilon)$ each satisfy the shear-free condition \eqref{FirstShearFree};

\item each initial data set in the family is  smoothly conformally compact, in the sense that $\bar g_\epsilon = \rho^2 g_\epsilon \in C^\infty(\bar M)$ and $\bar\Sigma_\epsilon = \rho(K_\epsilon + g_\epsilon) \in C^\infty(\bar M)$; and

\item we have $(g_\epsilon, K_\epsilon) \to (g,K)$ as $\epsilon \to 0$ in the $C^{k,\alpha}(M) \times C^{k,\alpha}(M)$ topology, for any $k$ and $\alpha$.
\end{enumerate}
\end{theorem}

We now describe the proof of Theorem \ref{Density}; the details are contained in \S \ref{FreeData}--\S\ref{SolveConstraint} below.
First we construct a family of free data $(\lambda_\epsilon, \mu_\epsilon)$ for small $\epsilon >0$.
Our construction is such that the metrics $\lambda_\epsilon$ agree with $g$ away from a neighborhood of $\partial M$, but are smoothly conformally compact.
We also arrange that the  fields $\mu_\epsilon$ agree with $\Sigma$ away from a neighborhood of $\partial M$, but are deformed near the boundary in order that the shear-free condition holds upon  deformation to a solution of the constraint equations.
The free data $(\lambda_\epsilon, \mu_\epsilon)$ is furthermore carefully constructed so that application of the conformal method yields smoothly conformally compact initial data sets.
(The construction is motivated by the analysis in \cite{AnderssonChrusciel-Obstructions}.)
The proof proceeds by applying the conformal method to the free data $(\lambda_\epsilon, \mu_\epsilon)$.
In order to show that the resulting solutions to the constraint equations approach $(g,K)$ as $\epsilon \to 0$, it is necessary to obtain uniform estimates for family of solutions $W_\epsilon$, $\phi_\epsilon$ to \eqref{ConformalConstraints}.

%------
\section{Technical preliminaries}

We present several technical results needed for the proof of Theorem \ref{Density}.

%----------
\subsection{Function spaces}
We fix a smooth defining function $\rho$ on $\bar M$, and we make use of weighted H\"older spaces $C^{k,\alpha}_\delta(M)$ of tensor fields on $M$ as defined in \cite{WAH-Preliminary} (see also \cite{Lee-FredholmOperators}).
These spaces are defined independently of any Riemannian structure, but have equivalent norms determined by any sufficiently regular asymptotically hyperbolic metric.
We emphasize that the convention regarding the weight $\delta$ is such that tensor field $u\in C^{0}_\delta(M)$ when $|u|_g \leq C \rho^\delta$ for any asymptotically hyperbolic metric $g = \rho^{-2}\bar g$.
Recall that the \Defn{weight of a tensor bundle} is the covariant rank less the contravariant rank.
(Thus the weight of a vector field is $-1$, while the weight of a metric tensor is $2$.)
The weight of a tensor field is important to keep in mind: If $u$ is a tensor field of weight $r$, then $|u|_g = \rho^{-r}|u|_{\bar g}$. In particular, for tensors of weight $r$ we have the following inclusion:
$$C^{k,\alpha}(\bar M)\subseteq C^{k,\alpha}_r(M);$$
compare with Lemma 3.7 of \cite{Lee-FredholmOperators}.

It is convenient to distinguish the following class of metrics: 
We say that an asymptotically hyperbolic metric $h$ is a \Defn{preferred background metric} if $\bar h =\rho^2 h$ extends smoothly to $\bar M$ and if in a neighborhood of $\partial M$ we have that $\bar h$ is a product metric of the form $\D\rho \otimes \D\rho +\bar b$ for some metric $\bar b$  on $\partial M$.
We denote by $\nabla$ and $\bar\nabla$ the Levi-Civita connections associated to $h$ and $\bar h$ respectively, and note that the difference tensor $\bar\nabla-\nabla$ is an element of $C^{k,\alpha}(M)$ for all $k$ and $\alpha$. 
Throughout this section and the next,  $h$ represents any preferred background metric, and $\bar h = \rho^2 h$.
In \S\ref{FreeData} we fix a preferred background metric, adapted to the metric $g$ appearing in Theorem \ref{Density}, which we retain  throughout the proof of that theorem.

The following is an immediate consequence of Lemma 2.2(d) of \cite{WAH-Preliminary}.
\begin{lemma}
\label{BoundaryLemma}
Suppose $u\in C^{k,\alpha}_r(M)$ is a tensor field of weight $r$ such that $\bar\nabla u\in C^{k-1,\alpha}_{r+1}(M)$ and such that $|u|_{\bar h}\to 0$ as $\rho \to 0$.
Then $u\in C^{k,\alpha}_{r+1}(M)$ and 
\begin{equation*}
\|u\|_{C^{k,\alpha}_{r+1}(M)} 
\leq C \left( 
\|u\|_{C^{k,\alpha}_r(M)}
+
\|\bar\nabla u\|_{C^{k-1,\alpha}_{r+1}(M)}
\right)
\end{equation*}
\end{lemma}

%--------
\subsection{Differential operators}
\label{DiffOpSec}

We now record several results concerning differential operators arising in the conformal method.
A differential operator $\mathcal P = \mathcal P[g]$ of order $l$ arising from a metric $g$ is said to be \Defn{geometric} (in the sense of \cite{Lee-FredholmOperators}) if in any coordinate frame the components of $\mathcal P u$ are linear functions of $u$ and its derivatives, whose coefficients are universal polynomials in the components of $g$, their partial derivatives, and $\sqrt{\det g_{ij}}$, such that the coefficient of the $j$th derivative of $u$ involves no more than $l-j$ derivatives of the metric. Such operators are {\it uniformly degenerate}; the mapping properties of such operators have been studied in \cite{Mazzeo-Edge}, \cite{Lee-FredholmOperators}, \cite{AnderssonChrusciel-Dissertationes}.
Recently, in work \cite{WAH-Preliminary} with James Isenberg and John M.~Lee we have extended some of these results to the {\it weakly asympotically hyperbolic} setting. We recall here several results needed for the proof of Theorem \ref{Density}; the aforementioned works apply in much more general settings.

The following proposition allows us to compare corresponding operators arising from different metrics.

\begin{proposition}[Proposition 7.9 of \cite{AHEM-Preliminary}]
\label{prop:GeometricOperatorsContinuous}
Let $k\geq 0$, $\alpha\in [0,1)$, and $\delta\in \mathbb R$.
Suppose $g\in C^{k,\alpha}(M)$ is an asymptotically hyperbolic metric, and that $\mathcal P$ is a geometric operator of order $l\leq k$.
Then there exists $\epsilon_*>0$ and $C>0$ such that for any asymptotically hyperbolic metric $g^\prime\in C^{k,\alpha}(M)$ with $\|g - g^\prime\|_{C^{k,\alpha}(M)}\leq \epsilon_*$ we have
\begin{equation*}
\|\mathcal P[g]u - \mathcal P[g^\prime]u \|_{C^{k-l,\alpha}_\delta(M)}
\leq C \| g-g^\prime\|_{C^{k,\alpha}(M)} \|u\|_{C^{k,\alpha}_\delta(M)}
\end{equation*}
for all $u\in C^{k,\alpha}_\delta(M)$.
\end{proposition}

We now turn attention to elliptic geometric operators.
The operators arising in the results here satisfy the following.

\setcounter{assumption}{15}
\begin{assumption}
\label{Assume-P}
Suppose $(M,g)$ is an asymptotically hyperbolic manifold.
We assume $\mathcal P = \mathcal P[g]$ is a second-order linear elliptic operator acting on sections of a tensor bundle $E$.
Furthermore we assume that
 $\mathcal P$ is {geometric} in the sense defined above, and that 
 $\mathcal P$ is formally self-adjoint.
\end{assumption}

The mapping properties of operators satisfying Assumption \ref{Assume-P} can be understood by studying the \Defn{indicial map} $I_s(\mathcal P)$, defined for $s\in \mathbb C$ to be the bundle map 
\begin{equation*}
E\otimes \mathbb C\big|_{\partial M}
\to
E\otimes \mathbb C\big|_{\partial M}
\end{equation*}
given by $I_s(\mathcal P) \bar u = \rho^{-s} \mathcal P (\rho^s \bar u)\big|_{\rho =0}$; see \cite{Mazzeo-Edge}, \cite{Lee-FredholmOperators}.
The \Defn{characteristic exponents} of $\mathcal P$, which we denote by $\mathcal E$, are defined to be those values of $s$ for which $I_s(\mathcal P)$ has a non-trivial kernel at some point on $\partial M$.
In \cite{Lee-FredholmOperators} it is shown that, under Assumption \ref{Assume-P}, these exponents and their multiplicities are constant on $\partial M$, and agree with those associated to  the corresponding operator in the half-space model of hyperbolic space.
Furthermore, due to the self-adjointness of $\mathcal P$, the characteristic exponents are symmetric about the line $\Re{(s)} = 1-r$, where $r$ is the weight of $E$.
The \Defn{indicial radius} $R$ of $\mathcal P$ is defined to be the smallest number $R\geq 0$ such that $\Re{(s)} \leq 1-r+R$ for all $s\in \mathcal E$.

The importance of the indicial radius is the following result from \cite{Lee-FredholmOperators}:
If $(M,g)$ is asymptotically hyperbolic of class $C^{k,\alpha}$ with $\alpha \in (0,1)$, if Assumption \ref{Assume-P} is satisfied, and if  
there is a compact set $K\subset M$ and a constant $C>0$ such that
\begin{equation}
\|u\|_{L^2(M)}\leq C\|\mathcal P u\|_{L^2(M)} 
\quad\text{ for all }\quad
 u\in C^\infty_c(M\setminus K),
\end{equation}
then $\mathcal P\colon C^{k+2,\alpha}_\delta(M) \to C^{k,\alpha}_\delta(M)$ is Fredholm if and only if $|1-\delta|<R$.

The following proposition is a consequence of  
\cite[Proposition 6.3]{AHEM-Preliminary}, 
\cite[Proposition 6.1]{WAH-Preliminary}, and 
\cite[Lemma 5.6]{WAH-Preliminary} (see also 
\cite[Lemma 6.4]{Lee-FredholmOperators}); we emphasize that the results cited apply in much more general situations.

\begin{proposition}
\label{MappingProperties}
Suppose $(M,g)$ is an asymptotically hyperbolic $3$-manifold, and suppose that $g$ is smoothly conformally compact.
\begin{enumerate}
\item\label{Prop4a}
For each $k\geq 0$, $\alpha \in (0,1)$, and $\delta\in (-1,3)$ the vector Laplacian is an isomorphism
\begin{equation*}
L_g\colon C^{k+2,\alpha}_\delta(M) \to C^{k,\alpha}_\delta(M).
\end{equation*}
In particular, there exists a constant $C>0$ such that
\begin{equation*}
\|X \|_{C^{k+2,\alpha}_\delta(M)} \leq C \|L_gX\|_{C^{k,\alpha}_\delta(M)}
\end{equation*}
for all vector fields $X\in C^{k+2,\alpha}_\delta(M)$.

\item\label{Prop4b}
Let $k\geq 0$ and $\alpha\in (0,1)$.
Suppose $\kappa \in C^{k,\alpha}_\sigma(M)$ for some $\sigma >0$ and that $c$ is a constant satisfying $c>-1$ and $c+\kappa\geq 0$.
Then so long as 
\begin{equation*}
|\delta - 1|\leq \sqrt{1+c}
\end{equation*}
the map
\begin{equation*}
\Delta_g - (c+\kappa)\colon C^{k+2,\alpha}_\delta(M) \to C^{k,\alpha}_\delta(M)
\end{equation*}
is an isomorphism.
In particular, there exists a constant $C>0$ such that
\begin{equation*}
\|u \|_{C^{k+2,\alpha}_\delta(M)} \leq C \|\Delta_gu - (c+\kappa)u\|_{C^{k,\alpha}_\delta(M)}
\end{equation*}
for all functions $u\in C^{k+2,\alpha}_\delta(M)$.

Furthermore, if $w\in C^{0}_\delta(M)$ is such that $\Delta_g w - (c+\kappa)w \in C^{k,\alpha}_{\delta^\prime}(M)$ then $w\in C^{k,\alpha}_{\delta^\prime}(M)$ whenever $|\delta^\prime - 1|\leq \sqrt{1+c}$.
\end{enumerate}

\end{proposition}

%--------------
\subsection{The tensor $\B_{\bar g}(\rho)$}
Together with James Isenberg and John M.~Lee, we introduced in \cite{WAH-Preliminary} a conformally invariant version of the trace-free Hessian that is used in \cite{AHEM-Preliminary} to characterize the shear-free condition; we now recall its definition and basic properties.
Let
\begin{equation*}
A_{\bar g}(\rho) 
= \frac{1}{2} |\D\rho|_{\bar g} \Div_{\bar g}\left[ |\D\rho|_{\bar g}\grad_{\bar g}\rho\right].
\end{equation*}
We  define the tensor field $\B_{\bar g}(\rho)$ by 
\begin{equation}
\label{DefineB}
\B_{\bar g}(\rho)
:=|\D\rho|_{\bar g}^6\,\mathcal D_{\bar g}(|\D\rho|^{-2}_{\bar g} \grad_{\bar g}\rho)
+ A_{\bar g}(\rho) \left( \D\rho \otimes \D\rho - \frac{1}{3}|\D\rho|^2_{\bar g}\, \bar g \right),
\end{equation}
where $\mathcal D_{\bar g}$ is the conformal Killing operator defined in \eqref{DefineCK}.

We have the following basic properties of $\B_{\bar g}(\rho)$.
\begin{proposition}[Proposition 4.1 of \cite{WAH-Preliminary}]\hfill
\label{B-BasicProperties}
\begin{enumerate}
\item $\B_{\bar g}(\rho)$ is symmetric and trace-free.

\item\label{B-TransverseProperty} $\B_{\bar g}(\rho)(\grad_{\bar g}\omega, \cdot)=0$.

\item\label{B-Scaling} $\B_{\bar g}(c \rho)=c^5\B_{\bar g}(\rho)$ for all constants $c$.

\item\label{B-ConformalScaling} If $\theta$ is a strictly positive function then 
$\B_{\theta^4\bar g}(\rho)=\theta^{-8}\B_{\bar g}(\rho)$
and $A_{\theta^4\bar g}(\rho) = \theta^{-8}A_{\bar g}(\rho)$.
\end{enumerate}
\end{proposition}

\begin{proposition}[Corollary 4.4 of \cite{AHEM-Preliminary}]
Suppose $(g,K) = (g, \Sigma - g)$ is a polyhomogeneous asymptotically-hyperbolic CMC initial data set.
Then the shear-free condition \eqref{FirstShearFree} is satisfied if and only if
\begin{equation}
\left.\bar\Sigma\right|_{\partial M} = \left.\B_{\bar g}(\rho)\right|_{\partial M},
\end{equation}
where $\bar g = \rho^2 g$ and $\bar\Sigma = \rho\Sigma$.

\end{proposition}

%---------------
\section{Analysis on $\bar M$}
\label{section:boundary}
The solution to a geometric elliptic equation of the form $\mathcal P u = f$ on an asymptotically hyperbolic manifold $(M,g)$ may be smooth on $M$, but may not extend smoothly to $\bar M$, even if $\bar g\in C^\infty(\bar M)$; see \cite{Mazzeo-Edge}, \cite{AnderssonChrusciel-Dissertationes}, \cite{ChruscielDelayLeeSkinner}, et.~al.
Rather, many solutions to elliptic equations have asymptotic expansions at $\partial M$ containing powers of $\rho$ and powers of $\log\rho$.
The logarithmic terms arise in situations where there is a resonance (see \S\ref{bdry:sec} or \cite[Remark A.12]{WAH-Preliminary}) and are thus features of the algebraic structure of $\mathcal P$.
Tensor fields with expansions involving powers of $\rho$ and $\log\rho$ are called {\it polyhomogeneous}.
We now present a more careful definition, and subsequently discuss conditions under which the solution itself is in fact smooth on $\bar M$.
We note that a number of related definitions of polyhomogeneity appear in the literature; see \cite{Mazzeo-Edge}, \cite{AnderssonChrusciel-Dissertationes}, \cite{WAH-Preliminary}, \cite{IsenbergLeeStavrov-AHP}, \cite{ChruscielDelayLeeSkinner}, et.~al.

For convenience, we work with a fixed preferred background metric $h = \rho^{-2}\bar h$, denoting by $\bar\nabla$ the Levi-Civita connection associated to $\bar h$.
(The following, however, is independent of the choice of $h$.)
We subsequently make frequent and implicit use of the following construction:
If $E$ is a tensor bundle over $\bar M$ and $\bar u$ is a smooth section of $E\big|_{\partial M}$, we may extend $\bar u$ to the neighborhood of $\partial M$ by parallel transport along $\grad_{\bar h}\rho$; using a smooth cutoff function, the resulting tensor may be extended further to all of $\bar M$.
Furthermore, when working in the neighborhood of $\partial M$ where $\bar h = \D\rho\otimes \D\rho + \bar b$, we abuse notation by writing $\rho\partial_\rho$ for $\rho\bar\nabla_{\grad_{\bar h}\rho}$.

\subsection{Polyhomogeneity}
In order to carefully define polyhomogeneity for tensor fields, we first introduce for each $\delta
\in \mathbb R$ the class $\mathcal B_\delta(M)$ of tensor fields, defined by
\begin{equation*}
\mathcal B_\delta(M)
= \bigcap_{\substack{0\leq k\\ t<\delta}} C^k_t(M).
\end{equation*}
(The reader may wish to compare these spaces to the conormality spaces appearing in \cite{WAH-Preliminary} and the references therein.)

The importance of this definition is that if $s\in \mathbb C$ then $\rho^s \log\rho$ is contained in $\mathcal B_\delta(M)$ with $\delta = \Re{(s)}$, but is not of class $C^0_\delta(M)$.
Furthermore, $u\in \mathcal B_\delta(M)$ if and only if $(\log\rho) u\in \mathcal B_\delta(M)$.
If tensor $u$ of weight $r$ satisfies  $u\in \mathcal B_{\delta+r}(M)$ for some $\delta>0$, then $(\rho\partial_\rho)^ku$ vanishes at $\partial M$ for all $k\geq 0$.
The same holds for certain other fields, such as the functions $\rho^s(\log\rho)^{-n}$ with $\Re{(s)}=0$ and $n$ a positive integer.
Consequently, we obtain the following.

\begin{lemma}
\label{LinInd}
Suppose $E$ is a tensor bundle of weight $r$, and that $t_i\in \mathbb{C}$, $q_i\in \mathbb{N}_0$, and sections $\bar u_i$ of $E\big{|}_{\partial M}$ are such that 
$$
u=\sum_{i=1}^N \rho^{t_i}(\log \rho)^{q_i}\bar u_i \in \mathcal B_{\delta+r}(M)
$$
for some $\delta>\max_{1\le i\le N}\Re(t_i)$. 
Then $\bar u_i=0$ for all $1\le i\le N$. 
\end{lemma}

\begin{proof}
It suffices to fix a point of $\partial M$ and consider the case when $u$ is a function. Under such restrictions our claim is a consequence of the fact that a finite $\mathbb{R}$-linear combination of single-variable functions of the form $\rho^t(\log \rho)^q$ with $\Re(t)=0$
$$\sum_{j=1}^{J} a_j \rho^{ib_j}(\log \rho)^{q_j}$$
vanishes at $\rho=0$ together with all of its $\rho\partial_\rho$ derivatives if and only if all of the coefficients $a_j$ vanish.
\end{proof}

A smooth section $u$ of tensor bundle $E$ over $M$ having weight $r$ is defined to be \Defn{polyhomogeneous} if 
\begin{enumerate}
\item there exist sequences $s_i\in \mathbb C$ and $p_i\in \mathbb N_0$ with $\Re{(s_i)}$ non-decreasing and diverging to $+\infty$ as $i\to\infty$,

\item 
\label{part:phg-bdy-sections}
for $i,p\in \mathbb N_0$ with $0\leq p \leq p_i$ there exists smooth section $\bar u_{ip}$ of $E\big|_{\partial M}$, 
and

\item
\label{part:phg-expansion}
 for each $k\in \mathbb N_0$ there exists $N_k\in \mathbb N_0$ such that 
\begin{equation}
\label{GenericTensorExpansion}
u - \sum_{i=0}^{N_k} \sum_{p=0}^{p_i} \rho^{s_i-r} (\log\rho)^p \bar u_{ip} \in \mathcal B_k(M).
\end{equation}

\end{enumerate}
We assume that those exponents $s_i$ having the same real part are ordered such that their imaginary parts are increasing.
(The factor of $\rho^{-r}$ in \eqref{part:phg-expansion} is motived by the fact that $|u|_g = \rho^r |u|_{\bar g}$; thus the leading order behavior of $|u|_{ g}$ will be as $\rho^{\Re{(s_0)}}(\log\rho)^{p_0}$.)

If $u$ satisfies the definition above, we write
\begin{equation*}
u\sim \sum_{i=0}^{\infty} \sum_{p=0}^{p_i} \rho^{s_i-r} (\log\rho)^p \bar u_{ip}.
\end{equation*}
Let $\mathcal B_\phg(M)$ be the collection of all tensor fields on $M$ which are polyhomogeneous as defined above.
We furthermore denote by $\mathcal B^\phg_\delta(M)$ those polyhomogeneous tensor fields that are of class $\mathcal B_\delta(M)$, and by $C^{k}_\phg(\bar M)$ those polyhomogeneous tensor fields extending to tensor fields of class $C^k$ on $\bar M$.

\begin{remark}\label{PhgSplitting}\hfill

\begin{enumerate}
\item It follows from Lemma \ref{LinInd} that if $u\in \mathcal B^\phg_\delta(M)$ then we have $\Re(s_0)\geq \delta$. 

\item Polyhomogeneous expansions are unique in the sense that if 
$$v\sim \sum_{i=0}^\infty \sum_{p=0}^{p_i} \rho^{s_i} (\log \rho)^p \bar v_{ip} \text{\ \ and \ \ } v\sim \sum_{i=0}^\infty \sum_{q=0}^{q_i} \rho^{t_i} (\log \rho)^p \bar w_{ip},$$
then $s_i=t_i$, $p_i=q_i$, and $\bar v_{ip}=\bar w_{jp}$.

\item 

Tensor fields $u$ which are smooth on $\bar M$  are polyhomogeneous with a Taylor-series like expansion 
$$
u\sim \sum_{n=0}^\infty \frac{\rho^n}{n!} \bar u_n.
$$
The fields $\bar u_n$ are the restrictions of $\bar\nabla {}^n u(\grad_{\bar h} \rho,\dots, \grad_{\bar h}\rho, \cdot, \dots, \cdot)$ to the boundary. 
We emphasize that this holds regardless of whether $u$ is analytic or not.

\item A tensor field $u\in \mathcal B^\phg_\delta(M)$ of weight $r$ is in $C^l_\phg(\bar M)$ if $\delta> l+r$; see Lemma 3.7 in \cite{Lee-FredholmOperators}.
Thus $u\in C^\infty(\bar M)$ if and only if $u\in C^\infty(\bar M)+\mathcal B^\phg_k(M)$ for all $k\in \mathbb{N}$. 
Furthermore, a polyhomogeneous tensor field 
$$u\sim \sum_{i=0}^\infty \sum_{p=0}^{p_i} \rho^{s_i} (\log \rho)^p \bar u_{ip}.$$
is smooth on $\bar M$ if and only if $s_i\in \mathbb{N}_0$ and $p_i=0$ for all $i$.
\end{enumerate}
\end{remark}

\subsection{PDE results}
The relationship between the uniformly degenerate elliptic operators and polyhomogeneity has been extensively studied in \cite{Mazzeo-Edge}; see also \cite{AnderssonChrusciel-Dissertationes}, \cite{WAH-Preliminary}, \cite{AHEM-Preliminary} for studies focusing on operators arising in the study of the Einstein constraint equations.
It this paper we make use of the following result, which is a consequence of Proposition 6.3 of \cite{AHEM-Preliminary} and Proposition 6.4 of \cite{WAH-Preliminary}.

\begin{proposition}
\label{SolveGeneric}
Suppose that $(M,g)$ is a smoothly conformally compact  asymptotically hyperbolic $3$-manifold. 
\begin{enumerate}

\item 
\label{SolveGenericVL}
If $Y$ is a vectorfield on $M$ which extends smoothly to $\bar M$, then the solution $W$ to 
\begin{equation*}
L_g W = Y
\end{equation*}
satisfies $W\in\rho^3 C^0_\phg(\bar M)$ and $\mathcal D_gW \in C^0_\phg(\bar M)$.

\item
\label{SolveGenericLich}
For any function $A\in \rho^2 C^\infty(\bar M)$,  there exists a unique positive solution $\phi\in C^2_\phg(\bar M)$ to
\begin{equation*}
\Delta_g\phi = \frac18 R[g] \phi - A \phi^{-7} + \frac{3}{4}\phi^5,
\qquad \left.\phi\right|_{\partial M} =1.
\end{equation*}
Furthermore, if $R[g]+6 = \mathcal O(\rho^2)$ then $\phi -1= \mathcal O(\rho^2)$.

\end{enumerate}
\end{proposition}

\subsection{Boundary regularity}\label{bdry:sec}
Even if $g$ is smoothly conformally compact and $f$ extends smoothly to $\bar M$, solutions to $\mathcal P[g]u = f$ may not extend smoothly to $\bar M$.
To understand why this is the case, and to understand those circumstances where $u$ {\it does} extend smoothly to $\bar M$, we examine more closely the relationship between $\mathcal{P}$ and its indicial map $I_s(\mathcal{P})$.
For a more general treatment of the subject the reader is referred to \cite{Mazzeo-Edge}; see also 
\cite{AnderssonChrusciel-Dissertationes}, 

In background coordinates $(\rho,\theta^1,\theta^2)=(\theta^0,\theta^1,\theta^2)$ near $\partial M$ (see \cite{WAH-Preliminary}, \cite{Lee-FredholmOperators}), we have 
\begin{equation*}
\mathcal P = a^{ij}(\rho\partial_i)(\rho\partial_j) + b^i(\rho\partial_i) + c, 
\end{equation*}
where the matrix-valued functions $a^{ij}$, $b^i$, and $c$ extend smoothly to $\rho=0$.
Computing in these coordinates one sees that 
$$I_s(\mathcal P) \bar u = \rho^{-s} \mathcal P (\rho^s \bar u)\big|_{\rho =0}=(\bar a^{\rho\rho} s^2
+ \bar b^\rho s+ \bar c) \bar u,
$$
where $\bar a^{\rho\rho} = a^{\rho\rho}\big|_{\rho =0}$, $\bar b^\rho=  b^\rho\big|_{\rho =0}$, and $\bar c = c\big|_{\rho =0}$ are smooth (matrix-valued) functions of $(\theta^1,\theta^2)$.

As in \cite{Mazzeo-Edge}, we define the \Defn{indicial operator} $I(\mathcal P)$ to be the unique dilation-invariant operator on $\partial M\times (0,\infty)$ such that
\begin{equation*}
I(\mathcal P) (\rho^s \bar u) = \rho^{s}I_s(\mathcal P)\bar u
\end{equation*}
for all smooth sections $\bar u$ of $E\big|_{\partial M}$.
Thus
\begin{equation}
\label{IndicialFull}
I(\mathcal P)(\rho^s (\log\rho)^p\bar u)
=
\sum_{k=0}^p \binom{p}{k}\rho^s (\log\rho)^{p-k} I_s^{(k)}(\mathcal P)\bar u,
\end{equation}
where $I_s^{(k)}(\mathcal P) = \frac{d^k}{ds^k}I_s(\mathcal P)$. In coordinates we have 
\begin{equation*}
I(\mathcal P) 
= \bar a^{\rho\rho}(\rho\partial_\rho)^2 
+ \bar b^\rho(\rho\partial_\rho) + \bar c,
\end{equation*}
with $\bar a^{\rho\rho}$, $\bar b^{\rho}$ and $\bar c$ as above. It should be noted that $I(\mathcal P)$ can be extended to a differential operator $\mathcal I(\mathcal P) u = I(\mathcal P)(\varphi u)$ on $M$ by means of a cut-off function $\varphi$ supported in a collar neighborhood of $\partial M$. We furthermore set $\mathcal R = \mathcal P - \mathcal I(\mathcal P)$.

Careful examinations of coordinate expressions for $\mathcal{P}$, $I(\mathcal P)$ and $\mathcal R$ yield the following: 

\begin{lemma}
\label{lemma:I-and-R}
Suppose $(M,g)$ is a smoothly conformally compact asymptotically hyperbolic manifold and that $\mathcal P$ satisfies Assumption \ref{Assume-P}.
Then for any $\delta\in \mathbb R$ we have
\begin{enumerate}
\item $\mathcal I(\mathcal P)\colon \mathcal B^\phg_\delta(M) \to \mathcal B^\phg_\delta(M)$ and
\item $\mathcal R\colon \mathcal B^\phg_\delta(M) \to \mathcal B^\phg_{\delta+1}(M)$.
\end{enumerate}
\end{lemma}

This  lemma can be interpreted as saying that $\mathcal I(\mathcal P)$ is an  approximation of $\mathcal{P}$ near $\partial M$. 
It is crucial to notice that $I(\mathcal P)$ is an operator of Cauchy-Euler type. The method advertised in entry-level courses for solving a constant coefficient Cauchy-Euler ODE such as 
\begin{equation}\label{Math235}
\bar a(\rho\partial_\rho)^2 u
+ \bar b(\rho\partial_\rho)u + \bar c u=f
\end{equation}
involves studying the roots $s_1$ and $s_2$ of the associated characteristic polynomial equation
$$\bar a s^2 + \bar b s + \bar c =0.$$
In the PDE setting, this corresponds to a study of characteristic exponents as defined in \S\ref{DiffOpSec}. 

Typical solutions to the ODE \eqref{Math235} have expansions in terms of powers of $\rho$, where the exponents present are the same as the exponents in the expansion of $f$, as well as the roots $s_i$.
However, when the expansion of $f$ includes $\rho^{s_i}$, we have a resonance  that leads to the presence of terms of the form $\rho^{s_i} \log\rho$ in the expansion of the solution $u$. 
Further resonances arise when $s_1=s_2$, in which case the two homogeneous solutions are $\rho^{s_1}$ and $\rho^{s_1}\log\rho$. 

The situation in the case of a  (self-adjoint, geometric, elliptic) PDE  in asymptotically hyperbolic setting is extremely similar to the ODE case.  
We now present conditions which ensure that no resonances, and thus no log terms, occur.
The proofs presented below are inspired by computations done in \cite{AnderssonChrusciel-Obstructions}.

\begin{proposition}
\label{prop:ForcingSmoothness}
Let $(M,g)$ be an asymptotically hyperbolic manifold that is smoothly conformally compact.
Suppose $\mathcal P = \mathcal P[g]$ acts on tensors of weight $r$ and satisfies Assumption \ref{Assume-P}, and let $\mu$ denote the maximum real part of the characteristic exponents of $\mathcal{P}$.
If $u\in \mathcal B_\phg(M)$ is such that 
\begin{enumerate}
\item $\mathcal P u$ extends to a tensor field in $C^\infty(\bar M)$, and
\item there exists $\delta >\mu$ such that $u\in C^\infty(\bar M)+\mathcal B_{\delta+r}^\phg(M)$,
\end{enumerate}
then $u$ extends to a tensor field in $C^\infty(\bar M)$.
\end{proposition}

\begin{proof}
Without loss of generality we may assume $u\in \mathcal B_{\delta+r}^\phg(M)$.
By Remark \ref{PhgSplitting} we may then assume that $u$  has polyhomogeneous expansion \eqref{GenericTensorExpansion} with  $\Re{(s_i - r)} >\mu$ for all $i$.
Let $\{\delta_j\}_{j=0}^\infty$ be the strictly increasing sequence listing the elements of $\Re\{s_i\}$. It suffices to show that for each $j\in \mathbb{N}_0$ there exists $\bar u_j \in C^\infty(\bar M)$ and $u_j \in \mathcal B^\phg_{\delta_j}(M)$ such that $u = \bar u_j + u_j$; we do so inductively. 

When $j=0$ there is nothing to prove as we may set $\bar u_0 =0$.
Thus we assume for some $j\geq 0$ that $u = \bar u_j + u_j$ as above.
Let
\begin{equation*}
w_j = \sum_{\Re{(s_i)}= \delta_j} \sum_{p=0}^{p_i} \rho^{s_i-r} (\log\rho)^p \bar u_{ip} 
\end{equation*}
and define $u_{j+1} = u_j - w_j$; note that $u_{j+1} \in \mathcal B^\phg_{\delta_{j+1}}(M)$ as desired. From Lemma \ref{lemma:I-and-R} we have that 
\begin{equation}
\label{IndicialAux}
\mathcal I(\mathcal P) w_j \in C^\infty(\bar M) + \mathcal B^\phg_{\delta^\prime}(M), \quad \delta^\prime >\delta_j.
\end{equation}
On the other hand, a direct computation in a collar neighborhood of the boundary $\partial M$ shows that 
$$
\mathcal I(\mathcal P) w_j
= \sum_{\Re{(s_i)}= \delta_j} \sum_{p=0}^{p_i} \rho^{s_i-r} (\log\rho)^p\bar w_{ip},
$$ 
where by \eqref{IndicialFull} we have 
\begin{equation}\label{IndicialNotSoFull}
\bar w_{ip_i}=I_{s_i-r}(\mathcal P) \bar u_{ip_i},\quad 
\bar w_{i(p_i-1)}=I_{s_i-r}(\mathcal P) \bar u_{i(p_i-1)}+p_i I_{s_i-r}^{(1)}(\mathcal P) \bar u_{ip_i},
\end{equation}
etc.

In view of Remark \ref{PhgSplitting} it follows that each exponent $s_i-r$ in the expansion of $w_j$ is a non-negative integer and that $I_{s_i-r}(\mathcal P) \bar u_{ip_i}=\bar w_{ip_i}= 0$ whenever $p_i\neq 0$. However, since $\Re{(s_i - r)} >\mu$ we can only have $I_{s_i-r}(\mathcal P) \bar u_{ip_i}= 0$ if $\bar u_{ip_i}=0$. Thus $p_i=0$, and the proof of our induction step is complete.
\end{proof}

For simplicity, we now restrict attention to a special class of operators, which includes those arising in the conformal method.

\setcounter{assumption}{11}
\begin{assumption}
\label{Assume-L}
Suppose $(M,g)$ is an asymptotically hyperbolic manifold, and that $\mathcal P=\mathcal P[g]$ is a geometric operator acting on sections of tensor bundle $E$ and satisfying Assumption \ref{Assume-P}.
We furthermore assume that the indicial operator $I_s(\mathcal P)$ is a product of a polynomial $p(s)$ and an isomorphism of $E\big|_{\partial M}$, where $p(s)$ has simple integer roots.
\end{assumption}

\begin{proposition}
\label{prop:Linear-L-Smooth}
Let $(M,g)$ be an asymptotically hyperbolic manifold that is smoothly conformally compact.
Suppose $\mathcal P = \mathcal P[g]$ acts on tensor field of weight $r$ and satisfies Assumption \ref{Assume-L}.
Let $\mu$ denote the highest characteristic exponent of $\mathcal{P}$.
If $u\in \mathcal B^\phg_{\mu+r}(M)$ satisfies $\mathcal P u \in C^\infty(\bar M)$ and $\mathcal Pu \in \mathcal B_{\delta+r}(M)$ for some $\delta>\mu\ge 0$, then $u$ extends to a smooth tensor field on $\bar M$.
\end{proposition}

\begin{proof}
Since $u\in \mathcal B^\phg_{\mu+r}(M)$, it admits an expansion  \eqref{GenericTensorExpansion} with $\Re{(s_i-r)}\geq \mu$.
Let
\begin{equation}\label{rndwexpansion}
w = \sum_{\Re{(s_i-r)}= \mu} \sum_{p=0}^{p_i} \rho^{s_i-r} (\log\rho)^p \bar u_{ip}.
\end{equation}
From Lemma \ref{lemma:I-and-R} we have 
$$\mathcal I(\mathcal P) w \in \mathcal B^\phg_{\delta^\prime+r}(M) \text{\ \ for some\ \ }\delta^\prime >\mu.$$
The computation \eqref{IndicialNotSoFull} and Remark \ref{PhgSplitting} now imply that $I_{s_i-r}(\mathcal P) \bar u_{ip_i}=0$
and, if $p_i\ge 1$, that
\begin{equation}\label{moreaux}
I_{s_i-r}(\mathcal P) \bar u_{i(p_i-1)}+p_i I_{s_i-r}^{(1)}(\mathcal P) \bar u_{ip_i}=0.
\end{equation}
Therefore, the only non-vanishing term in the expansion \eqref{rndwexpansion} has to correspond to $s_i-r=\mu$ which, by our assumptions, is a nonnegative integer. Furthermore, we must have $p_i=0$ because otherwise \eqref{moreaux} contradicts Assumption \ref{Assume-L}. Thus $w$ extends smoothly to $\bar M$ and our result is now immediate from Proposition \ref{prop:ForcingSmoothness}.
\end{proof}

We conclude this section with a regularity result for semilinear scalar equations of the form $\mathcal P u = f(u)$, where $f$ satisfies the following.

\setcounter{assumption}{5}
\begin{assumption}
\label{Assume-F}
We assume that $f$ is a smooth real function on $M\times I$ where $0\in I$ is an open interval. Furthermore, we assume that on a neighborhood of zero $f$ is given by an absolutely and uniformly convergent power series 
\begin{equation*}
f(x,u) = \sum_{l=0}^\infty a_l(x) u^l
\end{equation*}
with functions $a_0, a_1 \in \rho C^\infty(\bar M)$ and $a_l\in C^\infty(\bar M)$ for $l\geq 2$.
\end{assumption}
\noindent
In what follows we simply write $f(u)$ for $f(\cdot, u(\cdot))$.

\begin{remark}
\label{fcondition}
If $u\in \rho C^\infty(\bar M)+\mathcal{B}^\phg_\delta(M)$ with $\delta>1$, and if $f$ satisfies Assumption \ref{Assume-F} then $f(u)\in \rho C^\infty(\bar M)+\mathcal{B}^\phg_{\delta+1}(M)$.
\end{remark}

\begin{proposition}
\label{prop:SemilinearSmoothness}
Let $(M,g)$ be an asymptotically hyperbolic manifold that is smoothly conformally compact and suppose $\mathcal P = \mathcal P[g]$ is an elliptic operator acting on functions and satisfying Assumption \ref{Assume-L}.
Let $\mu$ denote the largest characteristic exponent of $\mathcal P$. 
Furthermore, let $f$ be a function satisfying Assumption \ref{Assume-F}. 

Suppose that $\mathcal P u = f(u)$, where $u \in \mathcal B^\phg_\mu(M)$ and $f(u) \in \mathcal B_\delta(M)$ for some $\delta>\mu$.
Then $u$ extends to a function in $C^\infty(\bar M)$.
\end{proposition}

\begin{proof}
Since $u\in \mathcal B^\phg_{\mu}(M)$, it admits an expansion  \eqref{GenericTensorExpansion} with $\Re{(s_i)}\geq \mu$. Note that $\mu \geq 1$, as a consequence of the fact that the set of the characteristic exponents of $\mathcal{P}$ is symmetric about $\Re{(s)}=1$ (cf.~Corollary 4.5 in \cite{Lee-FredholmOperators}). Furthermore, by Assumption \ref{Assume-L} we have that $\mu$ is an integer.

As in the proof of Proposition \ref{prop:Linear-L-Smooth} we consider the function
\begin{equation*}
w = \sum_{\Re{(s_i)}= \mu} \sum_{p=0}^{p_i} \rho^{s_i} (\log\rho)^p \bar u_{ip}.
\end{equation*}
From Lemma \ref{lemma:I-and-R} and the assumption that $f(u) \in \mathcal B_\delta(M)$ for some $\delta>\mu$ we have 
$$
I(\mathcal P)w\in \mathcal B^\phg_{\delta'}(M) 
\quad\text{ for some }\quad
 \delta' >\mu.
 $$
Arguing as in the proof of Proposition \ref{prop:Linear-L-Smooth} we obtain $s_i=\mu$ and $p_i=0$ for all $i$ in the above expression for $w$. Thus 
$$
u = \rho^\mu \bar u_{\mu 0} + v\in \rho C^\infty(\bar M)+\mathcal B^\phg_{\delta''}(M) 
\quad \text{ for some }\quad
\delta'' >\mu.
$$

It remains to establish smoothness of the function $v$. We do so by using the inductive argument from the proof of Proposition \ref{prop:ForcingSmoothness} to  show that for each $j$ there exist $\bar v_j \in \rho C^\infty(\bar M)$ and $v_j \in \mathcal B_{\delta_j}^\phg(M)$ such that $v = \bar v_j + v_j$. 
The inductive step relies on 
$$\mathcal I(\mathcal P) v_j = f(\bar u + \bar v_j + v_j) - \mathcal P\bar u - \mathcal P \bar v_j - \mathcal R v_j\in \rho C^\infty(\bar M)+\mathcal B^\phg_{\delta_j+1}(M),$$
which in turn is a consequence of Remark \ref{fcondition}.
\end{proof}

%---------------
\section{The free data}
\label{FreeData}
We now commence the proof of Theorem \ref{Density}, and assume that $(g,K)$ is a polyhomogeneous constant-mean-curvature asymptotically hyperbolic initial data set.

In this section we construct a family of free data $(\lambda_\epsilon, \mu_\epsilon)$, and subsequently establish several estimates for geometric quantities and differential operators associated to the family of metrics $\lambda_\epsilon$.
It is important that these estimates are uniform in $\epsilon>0$, in order that they lead to the convergence portion of Theorem \ref{Density}.
It is our convention that, unless otherwise stated, all constants are independent of $\epsilon$, provided $\epsilon$ is sufficiently small.

\subsection{Construction of the free data}
In order to construct a family of free data, we define a family of smooth cutoff functions.
Let $\chi\colon \mathbb R \to [0,1]$ be a smooth, decreasing function such that
\begin{equation*}
\chi(x) = 1\text{ if }x\geq 2
\quad \text{ and }\quad
\chi(x) = 0 \text{ if } x \leq 1.
\end{equation*}
For $\epsilon \in (0,1)$ define $\chi_\epsilon \colon M \to [0,1]$ by
$
\chi_\epsilon = \chi( { \rho}/{\epsilon}).
$
We note that 
$\supp \chi_\epsilon \subset \{\rho>\epsilon \}$ and that $\chi_\epsilon =1$ if $\rho\geq2\epsilon$.
Furthermore, $\D\chi_\epsilon = \chi^\prime(\rho/\epsilon) \epsilon^{-1}\D\rho$ is supported in $\{\epsilon \leq \rho \leq 2\epsilon\}$.
Thus,  since $\D\rho\in C^{k,\alpha}_1(M)$ for all $k\geq 1$ and $\alpha\in (0,1)$, we see that $\chi_\epsilon \in C^{k,\alpha}(M)$, with bound independent of $\epsilon$:
\begin{equation}
\label{chistuff}
\|\chi_\epsilon\|_{C^{k,\alpha}(M)}\leq C.
\end{equation}

Let $\bar b$ be the smooth metric induced on $\partial M$ by $\bar g = \rho^2 g$.
We define a preferred background metric $h$ by choosing $\bar h$ to be a smooth metric on $\bar M$ such that in a neighborhood of $\partial M$ we have 
\begin{equation}
\label{DefineBoundaryMetric}
\bar h = \D\rho \otimes \D\rho + \bar b
\end{equation}
and setting $h = \rho^{-2}\bar h$.
Let $\bar\nabla$ be the Levi-Civita connection associated to $\bar h$, and note that in the neighborhood of $\partial M$ where \eqref{DefineBoundaryMetric} holds we have $\Delta_{\bar h}\rho =0$.

We define, for sufficiently small $\epsilon>0$, the smooth metrics $\bar\lambda_\epsilon$ on $\bar M$ by 
\begin{equation}
\label{DefineLambdaBar}
\bar\lambda_\epsilon :=  \chi_\epsilon\, \bar g + (1-\chi_\epsilon) \bar h.
\end{equation}
Setting $\lambda_\epsilon = \rho^{-2}\bar\lambda_\epsilon$, we define the family of free data $(\lambda_\epsilon, \mu_\epsilon)$ by
\begin{equation*}
\lambda_\epsilon := \rho^{-2}\bar\lambda_\epsilon
\quad\text{ and }\quad
\mu_\epsilon:= \chi_\epsilon \Sigma = \chi_\epsilon \rho^{-1} \bar\Sigma.
\end{equation*}
We emphasize that $\lambda_\epsilon$ are each a smoothly conformally compact asymptotically hyperbolic metric on $M$.

\subsection{Estimates for $\lambda_\epsilon$}
We note the following properties of the metrics $\lambda_\epsilon$.
\begin{lemma}
\label{lemma:IntrinsicMetricEstimate}
Let $k\geq 0$ and $\alpha\in(0,1)$.  
We have
\begin{equation*}
\| g - \lambda_\epsilon\|_{C^{k,\alpha}_1(M)}\leq C
\quad
\text{ and }
\quad
\| g - \lambda_\epsilon\|_{C^{k,\alpha}(M)}\leq C\epsilon.
\end{equation*}
This furthermore implies that for sufficiently small $\epsilon>0$ we have
$\| g^{-1} - \lambda_\epsilon^{-1}\|_{C^{k,\alpha}(M)}\leq C\epsilon$. 
\end{lemma}

\begin{proof}
Since $\bar h$ agrees with $\bar g$ at $\rho=0$, we may apply Lemma \ref{BoundaryLemma} to conclude that
$\bar h - \bar g \in C^{k,\alpha}_3(M)$ with bound independent of $\epsilon$.
Also recall \eqref{chistuff}, which shows that the functions $1-\chi_\epsilon$ are uniformly bounded in $C^{k,\alpha}(M)$. Our first claim now follows from the identity $\lambda_\epsilon -  g = \rho^{-2}(1-\chi_\epsilon)(\bar h - \bar g)$. 

Since the support of $\lambda_\epsilon -  g$ is in $\{\rho \leq 2\epsilon\}$, the first estimate implies the second.
Finally, the estimate for the inverses comes from the second estimate applied to the series expansion $\lambda_\epsilon^{-1} -  g^{-1}$, centered at $g$.
\end{proof}

The following is immediate from the fact that $\bar\lambda_\epsilon=\bar h=\D\rho\otimes \D\rho+\bar b$ in a collar neighborhood of the boundary.

 \begin{lemma}
 \label{BoundaryH}
We have
$
\B_{\bar\lambda_\epsilon}(\rho)  =0
$
 and 
$
\lvert\D\rho\rvert^2_{\bar\lambda_\epsilon} =1
$
along $\partial M$.
 \end{lemma}

We now obtain estimates on the scalar curvature of the metrics $\lambda_\epsilon$.
We first note that
\begin{equation}
\label{FirstScalarLambda}
R[\lambda_\epsilon] +6 
= -6(|\D\rho|^2_{\bar\lambda_\epsilon} -1)
+4\rho \Delta_{\bar\lambda_\epsilon}\rho
+\rho^2 R[\bar\lambda_\epsilon].
\end{equation}

In a neighborhood of $\partial M$, where $\lambda_\epsilon = \rho^{-2}\bar h$, we have
\begin{equation}
\label{CompactifyR}
R[\lambda_\epsilon] + 6 = \rho^2 R[\bar h] \in C^{k,\alpha}_2(M),
\end{equation}
due to the fact that  
$|\D\rho|^2_{\bar h} \equiv 1$ and $\Delta_{\bar h}\rho \equiv 0$ near $\partial M$. 
However, we do not have a uniform estimate on $R[\lambda_\epsilon] + 6 $ in $C^{k,\alpha}_2(M)$.
Rather, we obtain the following.

\begin{proposition}
\label{prop:ScalarCurvature}
Let $k\geq 0$ and $\alpha\in(0,1)$. 
For sufficiently small $\epsilon>0$ we have
\begin{equation*}
\| R[\lambda_\epsilon] - R[g]\|_{C^{k,\alpha}_1(M)}\leq C
\quad\text{ and }\quad
\| R[\lambda_\epsilon] - R[g]\|_{C^{k,\alpha}(M)}\leq C\epsilon.
\end{equation*}
\end{proposition}

\begin{proof}
We make use of the formula \eqref{FirstScalarLambda}, analyzing each term on the right side.
The scalar curvature $R[\bar\lambda_\epsilon]$ is the sum of contractions of terms of the form
\begin{equation*}
(\bar\lambda_\epsilon)^{-1}\otimes
(\bar\lambda_\epsilon)^{-1}\otimes
(\bar\lambda_\epsilon)^{-1}\otimes
\bar\nabla\,\bar\lambda_\epsilon \otimes
\bar\nabla\,\bar\lambda_\epsilon
\quad\text{ and }\quad
(\bar\lambda_\epsilon)^{-1}\otimes
(\bar\lambda_\epsilon)^{-1}\otimes
\bar\nabla{}^2\bar\lambda_\epsilon;
\end{equation*}
The scalar curvature of $\bar g$ is comprised of analogous terms.
From Lemma \ref{lemma:IntrinsicMetricEstimate} we have
$$\|(\bar\lambda_\epsilon)^{-1} - (\bar g)^{-1}\|_{C^{k,\alpha}_{-2}(M)}
\leq \|\lambda_\epsilon^{-1} -  g^{-1}\|_{C^{k,\alpha}(M)}\leq C.$$ Likewise, both $\|\bar\nabla(\bar\lambda_\epsilon - \bar g)\|_{C^{k,\alpha}_3(M)}$ and $\|\bar\nabla^2(\bar\lambda_\epsilon - \bar g)\|_{C^{k,\alpha}_3(M)}$ can be bounded by 
$$\|\bar\lambda_\epsilon - \bar g\|_{C^{k+2,\alpha}_3(M)}
\leq
\|\lambda_\epsilon -  g\|_{C^{k+2,\alpha}_1(M)}\leq C.$$
We now conclude that 
\begin{equation*}
\|\rho^2(R[\bar\lambda_\epsilon]- R[\bar g])\|_{C^{k,\alpha}_{1}(M)} \leq C.
\end{equation*}
Similar reasoning, using that $\D\rho \in C^{k,\alpha}_1(M)$ and $\bar\nabla\D\rho\in C^{k,\alpha}_2(M)$, yields
\begin{equation*}
\|\rho(\Delta_{\bar\lambda_\epsilon}\rho- \Delta_{\bar g}\rho)\|_{C^{k,\alpha}_1(M)} \leq C.
\end{equation*}

Finally, we estimate the function 
$$
\eta = (|\D\rho|^2_{\bar\lambda_\epsilon}-1) - (|\D\rho|^2_{\bar g}-1) = ((\bar\lambda_\epsilon)^{-1} - (\bar g)^{-1})(\D\rho, \D\rho).
$$
Lemma \ref{lemma:IntrinsicMetricEstimate} implies that $\|\eta\|_{C^{2,\alpha}(M)}$ and $\|\bar\nabla \eta\|_{C^{1,\alpha}_1(M)}$ are uniformly bounded in $\epsilon$.
Since $\eta$ vanishes at $\rho =0$ we may apply Lemma \ref{BoundaryLemma} to conclude that $\|\eta\|_{C^{2,\alpha}_1(M)}$ is uniformly bounded in $\epsilon$.
This establishes the first estimate in the lemma. The second estimate follows from the first due to the fact that $\lambda_\epsilon$ agrees with $g$ for $\rho\geq 2\epsilon$.
\end{proof}

%----
\subsection{Estimates for geometric operators defined by $\lambda_\epsilon$}

Here we record several consequences of Proposition \ref{prop:GeometricOperatorsContinuous} and Lemma \ref{lemma:IntrinsicMetricEstimate}.

\begin{proposition}
\label{ell-stuff}
For any $k\geq 0$ and $\alpha\in [0,1)$, and for any $\delta\in \mathbb R$, there is a constant $C>0$, independent of sufficiently small $\epsilon$, such that the following hold:

\begin{enumerate}
\item\label{DivergenceEstimate}
For any tensor field $u\in C^{k+1,\alpha}_\delta(M)$ we have
\begin{equation*}
\| \Div_{\lambda_\epsilon} u \|_{C^{k,\alpha}_\delta(M)} \leq C \| u\|_{C^{k+1,\alpha}_\delta(M)}.
\end{equation*}

\item\label{Depsilon:est}
For any vector field $X\in C^{k+1,\alpha}_\delta(M)$ we have 
$$
\|\mathcal{D}_{\lambda_\epsilon}X\|_{C^{k,\alpha}_\delta(M)} \leq C \|X\|_{C^{k+1,\alpha}_\delta(M)}.
$$

\end{enumerate}
\end{proposition}

\begin{proof}
For the first claim, we note that the estimate holds with $\lambda_\epsilon$ replaced by $g$.
Since
\begin{equation*}
\| \Div_{\lambda_\epsilon}u\|_{C^{k,\alpha}_\delta(M)}
\leq \| \Div_{\lambda_\epsilon}u- \Div_{g}u\|_{C^{k,\alpha}_\delta(M)}
+\| \Div_{g}u\|_{C^{k,\alpha}_\delta(M)}
\end{equation*}
we may invoke Proposition \ref{prop:GeometricOperatorsContinuous} and Lemma \ref{lemma:IntrinsicMetricEstimate} to obtain the desired estimate.
The proof of the second claim follows from analogous reasoning.
\end{proof}

Due to Proposition \ref{MappingProperties} the vector Laplacian $L_{\lambda_\epsilon}:C^{k+2,\alpha}_\delta(M) \to C^{k,\alpha}_\delta(M)$ is invertible for each $\epsilon>0$, $k\geq 0$, $\alpha\in (0,1)$ and $\delta \in (-1,3)$. 
In particular, there exist constants $C_{\epsilon}$, depending on $\epsilon$, such that $\|X\|_{C^{k+2,\alpha}_\delta(M)} \leq C_{\epsilon} \| L_{\lambda_\epsilon}X\|_{C^{k,\alpha}_\delta(M)}$.
The linearized Licherowicz operator that appears in \S\ref{SolveConstraint} is similarly invertible for each $\lambda_\epsilon$. 
We now show that the invertibility estimates can be made uniform in $\epsilon$.

\begin{proposition}
\label{UniformInvertVL}

Let $k\geq 0$, $\alpha\in (0,1)$, and $\delta\in (-1,3)$. Furthermore, let the functions $\kappa, \kappa_\epsilon\in C^{k,\alpha}_1(M)$ be such that $\|\kappa_\epsilon-\kappa\|_{C^{k,\alpha}(M)}\leq C\epsilon$ and $3+\kappa\geq 0$.
Then there exists a constant $C>0$ such that:
\begin{enumerate}
\item for all vector fields $X\in C^{k+2,\alpha}_\delta(M)$ and for all sufficiently small $\epsilon >0$ we have
$$
\|X\|_{C^{k+2,\alpha}_\delta(M)} \leq C\| L_{\lambda_\epsilon}X\|_{C^{k,\alpha}_\delta(M)},
$$
and

\item for all functions $u\in C^{k+2,\alpha}_\delta(M)$ and for all sufficiently small $\epsilon >0$ we have
$$
\|u\|_{C^{k+2,\alpha}_\delta(M)} \leq C\| \Delta_{\lambda_\epsilon}u-(3+\kappa_\epsilon)u\|_{C^{k,\alpha}_\delta(M)}.
$$
\end{enumerate}
\end{proposition}

\begin{proof}
From Proposition \ref{MappingProperties}(\ref{Prop4a}) we have
\begin{equation}
\label{VL-Intermediate}
\begin{aligned}
\|X\|_{C^{k+2,\alpha}_\delta(M)}
&\leq C \|L_gX\|_{C^{k,\alpha}_\delta(M)}
\\
&\leq C\left(
\|L_gX-L_{\lambda_\epsilon}X\|_{C^{k,\alpha}_\delta(M)}
+\|L_{\lambda_\epsilon}X\|_{C^{k,\alpha}_\delta(M)}
\right).
\end{aligned}
\end{equation}
From Proposition \ref{prop:GeometricOperatorsContinuous} we have 
\begin{equation*}
\|L_gX-L_{\lambda_\epsilon}X\|_{C^{k,\alpha}_\delta(M)}
\leq C \|g-\lambda_\epsilon\|_{C^{k+2,\alpha}(M)} \|X\|_{C^{k+2,\alpha}_\delta(M)}.
\end{equation*} 
Making use of Lemma \ref{lemma:IntrinsicMetricEstimate}, we see that this term may be absorbed in to the left side of \eqref{VL-Intermediate} when $\epsilon>0$ is small; 
this proves the first invertibility estimate. The second estimate follows from a similar argument applied to  Proposition \ref{MappingProperties}(\ref{Prop4b}); the details are left to the reader.
\end{proof}

\subsection{Estimates for $\mu_\epsilon$}

\begin{lemma}
\label{EstimateMuEpsilon} Let $k\geq 0$ and $\alpha\in (0,1)$. 
There exists a constant $C>0$ such that
\begin{gather*}
\| \mu_\epsilon - \Sigma \|_{C^{k,\alpha}_1(M)}
\leq C,
\qquad
\| \mu_\epsilon - \Sigma \|_{C^{k,\alpha}(M)}
\leq C\epsilon,
\\
\|\Div_{\lambda_\epsilon} \mu_\epsilon  \|_{C^{k,\alpha}_1(M)}
\leq C,
\qquad
\| \Div_{\lambda_\epsilon}\mu_\epsilon \|_{C^{k,\alpha}(M)}
\leq C\epsilon.
\end{gather*}
Furthermore, $\Div_{\lambda_\epsilon}\mu_\epsilon \in C^{k,\alpha}_\delta(M)$ for all $\delta\in \mathbb R$.
\end{lemma}

\begin{proof}
First recall \eqref{chistuff} and note that $\bar\Sigma \in C^{1,\alpha}_2(M)$; thus $\mu_\epsilon = \rho^{-1}\chi_\epsilon \bar\Sigma$ is uniformly bounded in $C^{1,\alpha}_1(M)$, which implies the first estimate.
The second estimate follows from this and the fact that the support of $\mu_\epsilon - \Sigma$ is contained in the region where $\rho \leq 2\epsilon$.

The uniform bound on $\mu_\epsilon$ in  $C^{1,\alpha}_1(M)$, together with Proposition \ref{ell-stuff}\eqref{DivergenceEstimate}, implies that $\Div_{\lambda_\epsilon}\mu_\epsilon$ is uniformly bounded in $C^{0,\alpha}_1(M)$.
Since $\lambda_\epsilon$ agrees with $g$ and $\mu_\epsilon$ agrees with $\rho^{-1}\bar\Sigma$ for $\rho \geq 2\epsilon$, we see from \eqref{CMC-constraints} that $\Div_{\lambda_\epsilon}\mu_\epsilon$ is supported in the region $\epsilon \leq \rho \leq 2\epsilon$.
This, together with the third estimate, yields the fourth estimate.

Finally, the fact that $\mu_\epsilon$ is compactly supported implies that  $\Div_{\lambda_\epsilon}\mu_\epsilon \in C^{0,\alpha}_\delta(M)$ for all $\delta$.
\end{proof}

%-------------
\section{Construction of approximating initial data}\label{SolveConstraint}

\subsection{Analysis of the conformal momentum constraint}

For each free data set $(\lambda_\epsilon,\mu_\epsilon)$, Propositions \ref{MappingProperties} and \ref{SolveGeneric} guarantee that there exists a unique $W_\epsilon\in \rho^3 C^0_\phg(\bar M)$ such that
\begin{equation}
\label{VLepsilon}
L_{\lambda_\epsilon}W_\epsilon = -\Div_{\lambda_\epsilon}\mu_\epsilon
\end{equation}
and $\mathcal D_{\lambda_\epsilon} W_\epsilon \in C^0_\phg(\bar M)$.
By Proposition \ref{UniformInvertVL} there is a constant $C$ such that for all sufficiently small $\epsilon >0$ we have
\begin{equation*}
\|W_\epsilon\|_{C^{k+2,\alpha}_\delta(M)} \leq C \| \Div_{\lambda_\epsilon}\mu_\epsilon\|_{C^{k,\alpha}_\delta(M)}
\end{equation*}
for $\delta=0,1$. The estimates for $\mu_\epsilon$ in Lemma \ref{EstimateMuEpsilon} now imply the following estimates for the solutions $W_\epsilon$ to \eqref{VLepsilon}.

\begin{lemma}
\label{EstimateWepsilon}
Let $k\geq 0$ and $\alpha\in(0,1)$.
 There exists a constant $C>0$ such that 
\begin{gather*}
\| W_\epsilon \|_{C^{k,\alpha}_1(M)} \leq C,
\qquad
\| W_\epsilon \|_{C^{k,\alpha}(M)} \leq C\epsilon,
\\
\| \mathcal D_{\lambda_\epsilon}W_\epsilon \|_{C^{k,\alpha}_1(M)} \leq C,
\qquad
\| \mathcal D_{\lambda_\epsilon} W_\epsilon \|_{C^{k,\alpha}(M)} \leq C\epsilon.
\end{gather*}
\end{lemma}

We define the tensors $\sigma_\epsilon$ by
\begin{equation*}
\sigma_\epsilon := \mu_\epsilon + \mathcal D_{\lambda_\epsilon} W_\epsilon
\end{equation*}
and record the following consequence of Lemmas \ref{lemma:IntrinsicMetricEstimate}, \ref{EstimateMuEpsilon}, and \ref{EstimateWepsilon}.

\begin{lemma}
\label{EstimateSigmaNorms}
Let $k\geq 0$, $\alpha\in (0,1)$ and let $\epsilon>0$ be sufficiently small.
The function $|\sigma_\epsilon|^2_{\lambda_\epsilon}$ is in $C^{k,\alpha}_2(M)$, and satisfies
\begin{equation*}
\| |\sigma_\epsilon|^2_{\lambda_\epsilon} - |\Sigma|^2_{\lambda_\epsilon} \|_{C^{k,\alpha}_2(M)}
\leq C
\quad\text{ and }\quad
\| |\sigma_\epsilon|^2_{\lambda_\epsilon} - |\Sigma|^2_{\lambda_\epsilon} \|_{C^{k,\alpha}_1(M)}
\leq C\epsilon.
\end{equation*}
\end{lemma}

Finally, we address smoothness of the tensor fields $\sigma_\epsilon$. 
The strategy is to employ Proposition \ref{prop:Linear-L-Smooth} of \S\ref{section:boundary}.

\begin{proposition}\label{VL-NoLogs}
The solution $W_\epsilon$ of \eqref{VLepsilon} and the tensor field $\sigma_\epsilon$ extend smoothly to $\bar M$.
\end{proposition}

\begin{proof}
A direct computation shows that the indicial map $I_s(L_{\lambda_\epsilon})$ is
\begin{equation*}
I_s(L_{\lambda_\epsilon}) Y = -\frac12\left( Y + \frac13 Y(\rho) \grad_{\bar\lambda_\epsilon}\rho \right)(s^2 - 4s).
\end{equation*}
Thus $L_{\lambda_\epsilon}$ satisfies Assumption \ref{Assume-L} with the highest characteristic exponent of $\mu=4$. 
By Proposition \ref{SolveGeneric} we have that $W_\epsilon$ is polyhomogeneous, while Proposition \ref{MappingProperties} implies $W_\epsilon\in C^k_\delta(M)$ for all $k\geq 0$ and $\delta<3$. Since $L_{\lambda_\epsilon}W_\epsilon=-\Div_{\lambda_\epsilon}\mu_\epsilon$ extends smoothly to $\bar M$ and since by Lemma \ref{EstimateMuEpsilon} $\Div_{\lambda_\epsilon}\mu_\epsilon\in C^k_\delta(M)$ for all $k\geq 0$ and all $\delta\in \mathbb{R}$, we are in position to apply Proposition \ref{prop:Linear-L-Smooth}. Consequently, $W_\epsilon$ extends smoothly to $\bar M$, and thus $\sigma_\epsilon$ does as well.
\end{proof}

%---------------
\subsection{Analysis of the Lichnerowicz equation}
\label{AnalyzeLich}

From Proposition \ref{SolveGeneric}\eqref{SolveGenericLich} there exists, for each sufficiently small $\epsilon>0$, a {unique} positive polyhomogeneous function $\phi_\epsilon\in C^2_\phg(\bar M)$ such that
\begin{equation}
\label{LichEpsilon}
\begin{gathered}
0=\mathcal N_\epsilon(\phi_\epsilon):= \Delta_{\lambda_\epsilon}\phi_\epsilon 
-\frac18 R[\lambda_\epsilon] \phi_\epsilon + \frac18|\sigma_\epsilon |^2_{\lambda_\epsilon} \phi_\epsilon^{-7} - \frac{3}{4}\phi_\epsilon^5,
\\
\qquad \left.\phi_\epsilon\right|_{\partial M} =1.
\end{gathered}
\end{equation}

In order to obtain estimates on $\phi_\epsilon-1$ we first show that the constant function $\phi=1$ is an approximate solution of \eqref{LichEpsilon}.

\begin{lemma}
\label{N1-Estimate} 
Let $k\geq 0$ and $\alpha\in (0,1)$.
For each sufficiently small $\epsilon>0$ we have $\mathcal N_\epsilon(1)\in C^{0,\alpha}_1(M)$ with
\begin{equation}
\label{N1Estimate}
\|\mathcal N_\epsilon(1)\|_{C^{k,\alpha}_1(M)} \leq C
\quad\text{ and }\quad
\|\mathcal N_\epsilon(1)\|_{C^{k,\alpha}(M)} \leq C\epsilon.
\end{equation}
\end{lemma}

\begin{proof}
Using \eqref{CMC-constraints} we have 
\begin{equation*}
\mathcal N_\epsilon(1) = 
-\frac18\left(R[\lambda_\epsilon]-R[g]\right)
+  \frac18\left(|\sigma_\epsilon |^2_{\lambda_\epsilon}- |\Sigma|^2_{\lambda_\epsilon} \right)
+\frac18 \left(|\Sigma|^2_{\lambda_\epsilon}-|\Sigma|^2_{g}\right).
\end{equation*}
Estimates \eqref{N1Estimate} are now immediate from Proposition \ref{prop:ScalarCurvature}, Lemma \ref{EstimateSigmaNorms}, and 
Lemma \ref{lemma:IntrinsicMetricEstimate}.
\end{proof}

The linearization of $\mathcal N_\epsilon$ at $\phi =1$ is the operator
\begin{equation}
\label{LinearizedN}
\mathcal L_\epsilon := \Delta_{\lambda_\epsilon} - (3+\kappa_\epsilon),
\end{equation}
where 
\begin{equation*}
\kappa_\epsilon =  \frac18 \left( R[\lambda_\epsilon] + 6\right) + \frac78 |\sigma_\epsilon|^2_{\lambda_\epsilon}.
\end{equation*}

We now prove the properties of $\kappa_\epsilon$ needed in order to apply Proposition \ref{UniformInvertVL}. 
To that end we set $\kappa=|\Sigma|^2_g$ and
note that $\kappa\in C^{k,\alpha}_2(M)$ for all $k\geq 0$ and $\alpha\in (0,1)$.

\begin{lemma}
\label{kappa-BasicProperties}
For all $k\geq 0$ and $\alpha\in (0,1)$ and sufficiently small $\epsilon >0$ we have:
\begin{enumerate}
\item $\kappa_\epsilon \in C^{k,\alpha}_1(M)$.
\item $\|\kappa_\epsilon-\kappa\|_{C^{k,\alpha}(M)} \leq C\epsilon$.
\end{enumerate}
\end{lemma}

\begin{proof}
The fact that $\kappa_\epsilon \in C^{k,\alpha}_1(M)$ is immediate from Proposition \ref{prop:ScalarCurvature} and Lemma \ref{EstimateSigmaNorms}.
Using \eqref{CMC-constraints} we can express $\kappa_\epsilon-\kappa$ as
$$\kappa_\epsilon-\kappa=\frac18 \left( R[\lambda_\epsilon] -R[g]\right) + \frac78 \left(|\sigma_\epsilon|^2_{\lambda_\epsilon}-|\Sigma|^2_g\right).$$
The $C^{k,\alpha}(M)$ estimate on the difference of scalar curvatures follows from Proposition \ref{prop:ScalarCurvature}, while the remaining estimate follows from Lemmas \ref{lemma:IntrinsicMetricEstimate} and \ref{EstimateSigmaNorms}.
\end{proof}

We now see from Proposition \ref{UniformInvertVL} that for all  $\delta \in (-1,3)$ and sufficiently small $\epsilon>0$ the mapping
\begin{equation*}
\mathcal L_\epsilon \colon C^{k+2,\alpha}_\delta(M) \to C^{k,\alpha}_\delta(M)
\end{equation*}
defined in \eqref{LinearizedN} is an isomorphism with inverse bounded uniformly in $\epsilon$.

We now proceed to obtain estimates for $\phi_\epsilon$ by viewing $\phi_\epsilon -1$ as a fixed point of the map
\begin{equation*}
\mathcal G_\epsilon\colon u \to -\mathcal L_\epsilon^{-1}\left( \mathcal N_\epsilon(1) + \mathcal Q_\epsilon(u)\right),
\end{equation*}
where
\begin{equation*}
\begin{aligned}
\mathcal Q_\epsilon(u)
&:= \mathcal N_{\epsilon}(1+u) - \mathcal N_\epsilon(1) -\mathcal L_\epsilon(u)
\\
&= \frac18|\sigma_\epsilon|^2_{\lambda_\epsilon} \left( (1+u)^{-7} -1 + 7u\right)
-\frac34 \left( (1+u)^5 -1 -5u\right).
\end{aligned}
\end{equation*}
In preparation, we define for each $r_0, r_1>0$, which are not necessarily independent of $\epsilon$ and a fixed integer $k\in \mathbb{N}_0$, the collection of functions
\begin{equation*}
X(r_0, r_1) = \{ u \in C^{k+2,\alpha}_1(M) \mid 
\| u\|_{C^{k+2,\alpha}_1(M)} \leq r_1
 \text{ and }
\| u\|_{C^{k+2,\alpha}(M)} \leq r_0\}.
\end{equation*}
Note that for all $r_0, r_1>0$, the set $X(r_0, r_1)$ is a complete metric space with respect to  the norm $\gNorm{u}_X := \|u\|_{C^{k+2,\alpha}_1(M)} + \|u\|_{C^{k+2,\alpha}(M)}$.

We require the following mapping properties of $\mathcal Q_\epsilon$.
\begin{lemma}
\label{Q-Mapping}
Let $k\geq 0$ and $\alpha\in (0,1)$.
There exists $r_*>0$ and continuous function $F\colon [0,r_*]\times [0,r_*] \to [0,\infty)$, both independent of $\epsilon>0$,  such that $F(0,0) =0$ and such that for each $\delta\in [0,1]$ we have
\begin{equation*}
\| \mathcal Q_\epsilon(u) - \mathcal Q_\epsilon(v)\|_{C^{k,\alpha}_\delta(M)}
\leq F\left(\|u\|_{C^{k,\alpha}(M)} , \|v\|_{C^{k,\alpha}(M)}\right)\, \|u-v\|_{C^{k,\alpha}_\delta(M)}
\end{equation*}
for all $u,v\in X(r_0, r_1)$ with $r_0\in [0,r_*]$ and $r_1>0$.
In particular
\begin{equation*}
\| \mathcal Q_\epsilon(u) \|_{C^{k,\alpha}_\delta(M)} \leq F(r_0,0 ) \,\|u\|_{C^{k,\alpha}_\delta(M)}.
\end{equation*}
\end{lemma}

\begin{proof}
Set
\begin{equation*}
H_l(u,v) = u^{l-1}v + u^{l-2}v^2 + \dots + u v^{l-1}.
\end{equation*}
With
\begin{equation*}
Q_1(u):= u^5
\quad\text{ and }\quad
Q_2(u):=u^{-7}
\end{equation*}
we have
\begin{equation*}
Q_1(u) - Q_1(v) = (u-v) \sum_{l=2}^5 \binom{5}{l} H_l(u,v)
\end{equation*} 
and
\begin{equation*}
Q_2(u) - Q_2(v) = (u-v) \sum_{l=2}^\infty(-1)^l \binom{l+6}{l} H_l(u,v)
\end{equation*}
provided $|u|$ and $|v|$ are less than $1$.
The uniform bound on $|\sigma_\epsilon|^2_{\lambda_\epsilon}$ provided by Lemma \ref{EstimateSigmaNorms} implies that there are constants $C_*$ and $C_l$, $l\in \mathbb N$, such that
\begin{equation*}
F(u,v) = C_* \sum_{l=2}^\infty C_l H_l(u,v)
\end{equation*}
converges uniformly and has the desired properties.
\end{proof}

We now obtain the desired contraction property of $\mathcal G_\epsilon$.
\begin{lemma}
There exists constants $\epsilon_*>0$ and  $C_*>0$ such that for each $\epsilon\in (0,\epsilon_*]$ the map  $\mathcal G_\epsilon$ is a contraction mapping $X(C_*\epsilon, C_*) \to X(C_*\epsilon, C_*)$.
\end{lemma}

\begin{proof}
Let $u\in X(r_0, r_1)$ and $\delta\in \{0,1\}$.
By the uniform invertibility of $\mathcal L_\epsilon^{-1}$ (cf.~Proposition \ref{UniformInvertVL} and Lemma \ref{kappa-BasicProperties}) we have 
\begin{equation*}
\begin{aligned}
\| \mathcal G_\epsilon(u)\|_{C^{k+2,\alpha}_\delta(M)}
&= \| \mathcal L_\epsilon^{-1} \left( \mathcal N_\epsilon(1) + \mathcal Q_\epsilon(u)\right)\|_{C^{k+2,\alpha}_\delta(M)}
\\
&\leq C \|\mathcal N_\epsilon(1)\|_{C^{k,\alpha}_\delta(M)} 
+ C\|\mathcal Q_\epsilon(u)\|_{C^{k,\alpha}_\delta(M)}.
\end{aligned}
\end{equation*}
Using  Lemma \ref{N1-Estimate} and Lemma \ref{Q-Mapping} we obtain
\begin{gather*}
\| \mathcal G_\epsilon(u)\|_{C^{k+2,\alpha}(M)}
\leq C^\prime \epsilon + C^{\prime\prime} F(r_0, 0)\|u\|_{C^{k+2,\alpha}(M)}
\\
\| \mathcal G_\epsilon(u)\|_{C^{k+2,\alpha}_1(M)}
\leq C^\prime  + C^{\prime\prime} F(r_0, 0)\|u\|_{C^{k+2,\alpha}_1(M)}
\end{gather*}
for some constants $C^\prime,C^{\prime\prime}>0$ independent of $\epsilon$.
Choosing $C_*>2 C^\prime$ and, using the fact that $F(0,0)=0$, choosing $\epsilon_*$ small enough that $C^{\prime\prime}F(C_*\epsilon_*, 0)<1/2$ ensures that 
\begin{equation*}
\mathcal G_\epsilon\colon X(C_* \epsilon,C_*) \to X(C_* \epsilon,C_*)
\end{equation*}
whenever $\epsilon\in (0,\epsilon_*]$.

We furthermore have
\begin{equation*}
\| \mathcal G_\epsilon(u)- \mathcal G_\epsilon(v)\|_{C^{k+2,\alpha}_\delta(M)}
\leq C \| \mathcal Q_\epsilon(u) - \mathcal Q_\epsilon(v)\|_{C^{k,\alpha}_\delta(M)}.
\end{equation*}
Since $F(0,0)=0$, Lemma \ref{Q-Mapping} implies that we can choose $\epsilon_*>0$ such that $\mathcal G_\epsilon$ is a contraction for $\epsilon\in (0,\epsilon_*]$.
\end{proof}

The contraction property of $\mathcal G_\epsilon$, together with the Banach fixed point theorem, immediately leads to following.

\begin{proposition}
\label{PhiConverges}
Let $k\geq 0$ and $\alpha\in (0,1)$.
There exists $\epsilon_*>0$ and constant $C>0$ such that whenever $\epsilon \in (0,\epsilon_*)$ we have $\phi_\epsilon -1\in C^{k,\alpha}_1(M)$ and
\begin{equation*}
\| \phi_\epsilon -1\|_{C^{k,\alpha}(M)} 
\leq C \epsilon.
\end{equation*}
\end{proposition}

We now analyze the regularity on $\bar M$ of solutions $\phi_\epsilon$ of  \eqref{LichEpsilon}. 
We do so by writing \eqref{LichEpsilon} in terms of the auxiliary variable
$$u=\phi_\epsilon - u_0, 
\quad\text{ where }\quad
 u_0=1-\frac{1}{24}\rho^2 \R[\bar h].$$
This particular change of variable is motivated by the fact that, while Lemma \ref{N1-Estimate} shows that the function $1$ is an approximate solution to \eqref{LichEpsilon},  the function $u_0$ constitutes a better approximate solution. 
We make this precise in the following lemma. 

\begin{lemma}
\label{N2-Estimate} Let $k\geq 0$ and $\alpha\in (0,1)$.
For each sufficiently small $\epsilon>0$ we have $\mathcal N_\epsilon(u_0)\in \rho^4C^{\infty}(\bar M)$.
\end{lemma}

\begin{proof}
It suffices to perform the computation in the collar neighborhood of the boundary where $\lambda_\epsilon= h$. 
There we have $\R[h]=-6+\rho^2\R[\bar h]$ and $\langle \D\rho, \D\R[\bar h]\rangle_{\bar h}=0$; 
the latter can be seen as a consequence of the fact that $\grad_{\bar h}\rho$ is a Killing vector field in the collar neighborhood of $\partial M$. 
A direct computation now shows that 
$$\Delta_h u_0=\rho^2\Delta_{\bar h}u_0-\rho\langle \D\rho, \D u_0\rangle_{\bar h}=\rho^4\Delta_{\bar h}\R[\bar h]\in \rho^4 C^\infty(\bar M).$$
On the other hand, Propositions \ref{SolveGeneric} and \ref {VL-NoLogs} imply $|\sigma_\epsilon|^2_{\lambda_\epsilon} \in \rho^4 C^\infty(\bar M)$. Using this fact we obtain
\begin{multline*}
\frac{1}{8}\R[h]u_0-\frac{1}{8}|\sigma_\epsilon|^2_{h}u_0^{-7}+\frac{3}{4}u_0^5
\\
=\frac{1}{8}\left(-6+\frac{5}{4}\rho^2\R[\bar h]\right)+\frac{3}{4}\left(1-\frac{5}{24}\rho^2\R[\bar h]\right)+\rho^4 C^\infty(\bar M)\in \rho^4 C^\infty(\bar M).
\end{multline*}
This completes the proof.
\end{proof}

\begin{proposition}\label{LichNoLogs}
The solution $\phi_\epsilon$ of \eqref{LichEpsilon} extends to a smooth function on $\bar M$. 
\end{proposition}

\begin{proof}
We see from Lemma \ref{N2-Estimate} that 
$$\Delta_hu=\frac{1}{8}\R[h]u-\frac{1}{8}|\sigma_\epsilon|^2_h\left((u_0+u)^{-7}-u_0^{-7}\right)+\frac{3}{4}\left((u_0+u)^5-u_0^5\right)+\rho^4C^\infty(\bar M).$$
Since $\frac{1}{8}\R[h]=-\frac{3}{4}+\rho^2C^\infty(\bar M)$, $|\sigma_\epsilon|^2_{\lambda_\epsilon} \in \rho^4 C^\infty(\bar M)$, and $\frac{15}{4}u_0^4=\frac{15}{4}+\rho^2C^\infty(\bar M)$, we have
\begin{equation}
\label{NewLichU}
\Delta_{\lambda_\epsilon}u - 3 u = f(u),
\end{equation}
where near $\rho=0$ the function $f(u)$ has the uniformly and absolutely convergent power series
\begin{equation*}
f(u) = \sum_{l=0}^\infty a_l u^l
\end{equation*}
with $a_0 \in \rho^4 C^\infty(\bar M)$, $a_1\in \rho^2 C^\infty(\bar M)$, and $a_l \in C^\infty(\bar M)$ for $l\geq 2$. In particular, $f$ satisfies Assumption \ref{Assume-F}. Also note that, by Proposition \ref{SolveGeneric}\eqref{SolveGenericLich}, $u\in \rho^2C^0_\phg(\bar M)$; consequently $f(u)\in C^{k,\alpha}_4$ for all $k\geq 0$ and $\alpha\in(0,1)$. Applying Proposition \ref{MappingProperties} now yields $u\in C^k_\delta(M)$ for all $k\geq 0$ and $\delta<3$.

Finally, we observe that the indicial map of $\Delta_{\lambda_\epsilon} -3$ is
$$ I_s(\Delta_{\lambda_\epsilon} -3) = (s-3)(s+1).$$
In particular, $\Delta_{\lambda_\epsilon}-3$ satisfies Assumption \ref{Assume-L} with the highest characteristic exponent of $\mu=3$.  Invoking Proposition \ref{prop:SemilinearSmoothness} we conclude that $u$ and $\phi_\epsilon$ extend to functions in $C^\infty(\bar M)$.
\end{proof}

\subsection{The proof of Theorem \ref{Density}}
We now construct the approximating initial data and show that they satisfy the shear-free condition, are smoothly conformally compact, and have the desired convergence property.

The solutions $W_\epsilon$ to \eqref{VLepsilon} and $\phi_\epsilon$ to \eqref{LichEpsilon} give rise to initial data sets $(g_\epsilon, K_\epsilon)$ determined by 
\begin{equation}
\label{BuildApproximates}
\begin{aligned}
g_\epsilon &= \phi_\epsilon^4 \lambda_\epsilon
\\
K_\epsilon &= \Sigma_\epsilon - g_\epsilon = \phi_\epsilon^{-2} \sigma_\epsilon - \phi_\epsilon^4\lambda_\epsilon.
\end{aligned}
\end{equation}
By Propositions \ref{VL-NoLogs} and \ref{LichNoLogs} we see that $\bar g_\epsilon=\rho^2g_\epsilon$ and $\bar \Sigma_\epsilon=\rho(K_\epsilon+g_\epsilon)$ extend smoothly to $\bar M$.

To see that $(g_\epsilon, K_\epsilon)$ is shear-free note that  Lemma \ref{BoundaryH}, Proposition \ref{B-BasicProperties}, and the fact that $\phi_\epsilon =1$ along $\partial M$ imply 
\begin{equation*}
\left.\B_{\bar g_\epsilon}(\rho)\right|_{\partial M} =0.
\end{equation*}
In addition, we have
\begin{equation*}
\left.\bar\Sigma_{\epsilon}\right|_{\partial M}
= \left.\rho\sigma_\epsilon\right|_{\partial M}
=\left.\rho \left(\mu_\epsilon  + \mathcal D_{\lambda_\epsilon}W_\epsilon\right)\right|_{\partial M}.
\end{equation*}
By definition, $\mu_\epsilon$ vanishes along $\partial M$. 
Furthermore, Proposition \ref{SolveGeneric} implies that $\mathcal D_{\lambda_\epsilon}W_\epsilon \in C^0_\phg(\bar M)$, and thus we see that $\bar\Sigma_{\epsilon}$ vanishes along $\partial M$.
Consequently, the approximating family of initial data $(g_\epsilon, K_\epsilon)$ satisfies the shear-free condition.

Finally, we prove the following convergence property.
\begin{proposition}
Let $k\geq 0$ and $\alpha\in (0,1)$. Then 
$$\|g_\epsilon-g\|_{C^{k,\alpha}(M)}\leq C\epsilon, \ \ \|K_\epsilon-K\|_{C^{k,\alpha}(M)}\leq C\epsilon$$
for some constant $C$ independent of $\epsilon$.
\end{proposition}

\begin{proof}
We have
\begin{equation*}
g_\epsilon - g = \phi_\epsilon^4 (\lambda_\epsilon - g) + (\phi_\epsilon^4 -1) g.
\end{equation*}
From Lemma \ref{lemma:IntrinsicMetricEstimate} we see that the $C^{k,\alpha}$ norm of the first term is $\mathcal O(\epsilon)$, while the second term can be estimated using Proposition \ref{PhiConverges}.

Note that $K - K_\epsilon = \Sigma - \Sigma_\epsilon -(g-g_\epsilon)$.
Thus it suffices to estimate
\begin{equation*}
\Sigma - \Sigma_\epsilon
= \Sigma - \phi_\epsilon^{-2}(\mu_\epsilon + \mathcal D_{\lambda_\epsilon}W_\epsilon).
\end{equation*}
Due to Lemma \ref{EstimateWepsilon} and Proposition \ref{PhiConverges}, it suffices to estimate $\Sigma-\mu_\epsilon$.
This, however, is accomplished in Lemma \ref{EstimateMuEpsilon}.
\end{proof}

%--------------------------------
\bibliographystyle{plain}
\bibliography{ShearFreeDensity}

\begin{thebibliography}{10}

\bibitem{AHEM-Preliminary}
Paul~T. Allen, James Isenberg, John~M. Lee, and Iva Stavrov.
\newblock The shear-free condition and constant-mean-curvature hyperboloidal
  initial data.
\newblock In preparation.

\bibitem{WAH-Preliminary}
Paul~T. Allen, James Isenberg, John~M. Lee, and Iva Stavrov.
\newblock Weakly asymptotically hyperbolic manifolds.
\newblock \texttt{arXiv:1506.03399}.

\bibitem{AnderssonChrusciel-Obstructions}
Lars Andersson and Piotr~T. Chru{{\'s}}ciel.
\newblock Hyperboloidal {C}auchy data for vacuum {E}instein equations and
  obstructions to smoothness of null infinity.
\newblock {\em Phys. Rev. Lett.}, 70(19):2829--2832, 1993.

\bibitem{AnderssonChrusciel-Dissertationes}
Lars Andersson and Piotr~T. Chru{{\'s}}ciel.
\newblock Solutions of the constraint equations in general relativity
  satisfying ``hyperboloidal boundary conditions''.
\newblock {\em Dissertationes Math. (Rozprawy Mat.)}, 355:100, 1996.

\bibitem{AnderssonChruscielFriedrich}
Lars Andersson, Piotr~T. Chru{{\'s}}ciel, and Helmut Friedrich.
\newblock On the regularity of solutions to the {Y}amabe equation and the
  existence of smooth hyperboloidal initial data for {E}instein's field
  equations.
\newblock {\em Comm. Math. Phys.}, 149(3):587--612, 1992.

\bibitem{ChoquetBruhat-GRBook}
Yvonne Choquet-Bruhat.
\newblock {\em General relativity and the {E}instein equations}.
\newblock Oxford Mathematical Monographs. Oxford University Press, Oxford,
  2009.

\bibitem{ChruscielDelayLeeSkinner}
Piotr~T. Chru{{\'s}}ciel, Erwann Delay, John~M. Lee, and Dale~N. Skinner.
\newblock Boundary regularity of conformally compact {E}instein metrics.
\newblock {\em J. Differential Geom.}, 69(1):111--136, 2005.

\bibitem{Friedrich-ConformalFieldEquations}
Helmut Friedrich.
\newblock Cauchy problems for the conformal vacuum field equations in general
  relativity.
\newblock {\em Comm. Math. Phys.}, 91(4):445--472, 1983.

\bibitem{Friedrich-StaticRadiative}
Helmut Friedrich.
\newblock On static and radiative space-times.
\newblock {\em Comm. Math. Phys.}, 119(1):51--73, 1988.

\bibitem{IsenbergLeeStavrov-AHP}
James Isenberg, John~M. Lee, and Iva Stavrov~Allen.
\newblock Asymptotic gluing of asymptotically hyperbolic solutions to the
  {E}instein constraint equations.
\newblock {\em Ann. Henri Poincar{\'e}}, 11(5):881--927, 2010.

\bibitem{Lee-FredholmOperators}
John~M. Lee.
\newblock Fredholm operators and {E}instein metrics on conformally compact
  manifolds.
\newblock {\em Mem. Amer. Math. Soc.}, 183(864):vi+83, 2006.

\bibitem{Mazzeo-Edge}
Rafe Mazzeo.
\newblock Elliptic theory of differential edge operators. {I}.
\newblock {\em Comm. Partial Differential Equations}, 16(10):1615--1664, 1991.

\end{thebibliography}

\end{document}